\documentclass[hidelinks]{article}
\usepackage[utf8]{inputenc}

\usepackage{amsmath,amsfonts,amssymb,amsthm}
\usepackage[utf8]{inputenc}
\usepackage[T1]{fontenc}    
\usepackage{hyperref}       
\usepackage{url}            
\usepackage{booktabs}       
\usepackage{amsfonts}       
\usepackage{nicefrac}       
\usepackage{microtype}      
\usepackage{xcolor}         

\usepackage{amsmath}
\usepackage[nottoc,notlof,notlot]{tocbibind} 
\renewcommand{\thesection}{\arabic{section}}
\makeatletter
\let\l@section\l@chapter
\makeatother
\usepackage{amssymb}
\usepackage{amsthm}
\usepackage{hyperref} 
\usepackage{algorithm}
\usepackage{algorithmic}
\usepackage{url}
\usepackage{graphicx}
\usepackage{tcolorbox}
\usepackage{amsthm}
\usepackage{xcolor}
\usepackage{float}
\floatstyle{boxed} 
\usepackage{thm-restate}
\restylefloat{figure}
\usepackage{upgreek}
\usepackage{authblk}
\newtheorem{theorem}{Theorem}

\newtheorem*{remark}{Remark}
\newtheorem{corollary}[theorem]{Corollary}
\newtheorem{lemma}[theorem]{Lemma}
\newtheorem{Proposition}[theorem]{Proposition}

\newcommand{\bigO}{\mathcal{O}}
\DeclareMathOperator{\Tr}{Tr}
\DeclareMathOperator{\myspan}{span}
\DeclareMathOperator{\Gr}{Gr}
\DeclareMathOperator{\diag}{diag}
\DeclareMathOperator{\dist}{dist}
\DeclareMathOperator{\grad}{grad}
\DeclareMathOperator{\Hess}{Hess}
\DeclareMathOperator{\Exp}{Exp}
\DeclareMathOperator{\Log}{Log}
\DeclareMathOperator{\atan}{atan}
\DeclareMathOperator{\vecop}{vec}
\DeclareMathSymbol{\Gamma}{\mathord}{operators}{"00}
\DeclareMathSymbol{\Delta}{\mathord}{operators}{"01}
\DeclareMathSymbol{\Theta}{\mathord}{operators}{"02}
\DeclareMathSymbol{\Lambda}{\mathord}{operators}{"03}
\DeclareMathSymbol{\Xi}{\mathord}{operators}{"04}
\DeclareMathSymbol{\Pi}{\mathord}{operators}{"05}
\DeclareMathSymbol{\Sigma}{\mathord}{operators}{"06}
\DeclareMathSymbol{\Upsilon}{\mathord}{operators}{"07}
\DeclareMathSymbol{\Phi}{\mathord}{operators}{"08}
\DeclareMathSymbol{\Psi}{\mathord}{operators}{"09}
\DeclareMathSymbol{\Omega}{\mathord}{operators}{"0A}

\title{Geodesic Convexity of the Symmetric Eigenvalue Problem and Convergence of Riemannian Steepest Descent}
\date{}
\author{Foivos Alimisis  \hspace{4mm}  Bart Vandereycken }
\affil{Department of Mathematics, University of Geneva}
\begin{document}

\maketitle

\begin{abstract}    
We study the convergence of the Riemannian steepest descent algorithm on the Grassmann manifold for minimizing the block version of the Rayleigh quotient of a symmetric matrix. Even though this problem is non-convex in the Euclidean sense and only very locally convex in the Riemannian sense, we discover a structure for this problem that is similar to geodesic strong convexity, namely, weak-strong convexity. This allows us to apply similar arguments from convex optimization when studying the convergence of the steepest descent algorithm but with initialization conditions that do not depend on the eigengap $\delta$. When $\delta>0$, we prove exponential convergence rates, while otherwise the convergence is algebraic. Additionally, we prove that this problem is geodesically convex in a neighbourhood of the global minimizer of radius $\bigO(\sqrt{\delta})$.
\end{abstract}

\paragraph{Keywords:} Block Rayleigh quotient, Grassmann manifold, Geodesic convexity, Riemannian optimization, Low-rank approximation

\paragraph{Acknowledgements:} This work was supported by the SNSF under research project 192363.

\section{Introduction}

\hspace{4mm} We consider the problem of computing the top $k$ eigenvectors of a symmetric matrix $A \in \mathbb{R}^{n \times n}$, which has many applications in numerical linear algebra (low rank approximation), statistics (principal component analysis) and signal processing. Without loss of generality, we assume that $A$ is also positive semidefinite. This is because $A$ can be shifted as $A+c I_n$ for some constant $c$ and this transformation does not change its eigenvectors.

We denote by $\lambda_1 \geq \lambda_2 \geq \cdots \geq \lambda_n$ the eigenvalues of $A$ counted with multiplicity and by $\delta:=\lambda_k-\lambda_{k+1}$ the eigengap  for some $k$ between $1$ and $n-1$. We also denote $\Lambda_{\alpha} = \textnormal{diag}(\lambda_1, \ldots ,\lambda_k)$ and  $\Lambda_{\beta} = \textnormal{diag}(\lambda_{k+1}, \ldots ,\lambda_n)$.

A set of $k$ leading eigenvectors of $A$ can be found by minimizing the function
\[
 f(X) = -\Tr(X^TAX)
\]
over the set of $n \times k$ matrices with orthonormal columns. Indeed, from the Ky-Fan theorem we know that
\begin{equation}\label{eq:min_f_over_X}
  \min \{ f(X) \colon X \in \mathbb{R}^{n \times k}, X^T X = I_k \} = -(\lambda_1 + \cdots + \lambda_k)=-\Tr(\Lambda_{\alpha})=:f^*.
\end{equation}
Since $A$ is symmetric, we can define the matrix $V_{\alpha} = \begin{bmatrix} v_1 \hspace{2mm} \cdots \hspace{2mm} v_k \end{bmatrix}$ such that $V_{\alpha}^T V_{\alpha} = I_k$ and with $v_i \in \mathbb{R}^n$ a unit-norm eigenvector corresponding to $\lambda_i$. If the eigengap $\delta$ is strictly positive, then $\myspan(V_{\alpha})$ is unique; otherwise, we can choose any $v_k$ from a subspace with dimension equal to the multiplicity of $\lambda_k$. It is readily seen that $f(V_{\alpha})= -(\lambda_1 + \cdots + \lambda_k)$. In fact, all minimizers of~\eqref{eq:min_f_over_X} are of the form $V_{\alpha} Q$ with $Q$ a $k \times k$ orthogonal matrix. We also define $V_{\beta}=\begin{bmatrix} v_{k+1} \hspace{2mm} \cdots \hspace{2mm} v_n  \end{bmatrix}$ that contains the eigenvectors corresponding to the eigenvalues $\lambda_{k+1},\ldots,\lambda_n$. Its columns span the orthogonal complement of $\myspan(V_{\alpha})$ in $\mathbb{R}^n$ and thus $V_{\beta}^T V_{\beta} = I_{n-k}$ and $V_{\alpha}^T V_{\beta} = {0}_{k \times (n-k)}$.

Since $\myspan(V_{\alpha}) = \myspan(V_{\alpha}Q)$, it is more natural to consider this problem as a minimization problem on the Grassmann manifold $\Gr(n,k)$, the set of $k$-dimensional subspaces in $\mathbb{R}^n$. Let us therefore redefine the objective function as
\begin{equation}\label{eq:min_f_over_Gr}
 f(\mathcal{X}) = -\Tr(X^TAX) \text{ where $\mathcal{X} = \myspan(X)$ for $X \in \mathbb{R}^{n \times k}$ s.t. $X^T X = I_k$}.
\end{equation}
This cost function can be seen as a block version of the standard Rayleigh quotient.
An immediate benefit is that, if $\delta > 0$, the minimizer of~\eqref{eq:min_f_over_Gr} is isolated since it is the subspace $\mathcal{V_{\alpha}} = \myspan(V_\alpha)$.

To minimize $f$ on $\Gr(n,k)$, we shall use the Riemannian steepest descent method (RSD) along geodesics in $\Gr(n,k)$. Quite remarkably, for $\Gr(n,k)$ these geodesics can be implemented efficiently in closed form.

For analyzing the convergence properties of steepest descent on $\Gr(n,k)$, we extend results of the recent work \cite{alimisis2021distributed}, where it is shown that the Rayleigh quotient on the sphere enjoys favourable geodesic convexity-like properties, namely, \emph{weak-quasi-convexity} and \emph{quadratic growth}. In this work, we show that these convexity-like properties continue to hold in the more general case of the block Rayleigh quotient function $f\colon \Gr(n,k) \rightarrow \mathbb{R}$. These results are of general interest, but also sufficient to prove a local convergence rate for steepest descent for minimizing $f$ when started from an initial point outside the region of local convexity. For the latter, a crucial help is provided by the fact that the Grassmann manifold is positively curved.

In particular, assuming a \emph{strictly positive eigengap} $\delta$ between $\lambda_k$ and $\lambda_{k+1}$, we prove an exponential convergence rate to the subspace spanned by the $k$ leading eigenvectors, similar to the convergence of power method and subspace iteration (Theorem \ref{thm:exponential_conv}). If we do not assume any knowledge regarding the eigengap, then we can still prove a sub-exponential (polynomial) convergence rate of the function values to the global minimum (Theorem \ref{thm:convex_conv}), but we cannot directly study the convergence to a global minimizer. This is in line with  previous work but our analysis does not use standard notions of geodesic convexity and allows for an initial guess further from the global minimizer. In Appendix \ref{sec:big_step} we present related convergence results for steepest descent with a more tractable step size but at the expense of needing a slightly better initialization.

\section{Related work}
The symmetric eigenvalue problem has been popular for several decades in the numerical linear algebra and optimization communities. When only a few eigenvalues are targeted, the main solvers for this problem have been based on subspace iteration and Krylov subspace methods. Less but still considerable attention has been given to the steepest descent method and its accelerated versions. Most works on steepest descent focus only on computing the first leading eigenvector of a symmetric matrix ($k=1$), using a Euclidean version of the algorithm. Asymptotic convergence rates are known for this setting since the 1950's, see \cite{Hestenes}. More recently, exact non-asymptotic estimates for the same Euclidean steepest descent with exact line search were proved in \cite{KNYAZEV1991245}. For a more comprehensive overview of this line of research, the reader can refer to \cite{Neymer} and the references therein. A recent result that takes a different route compared to the previous ones is \cite{alimisis2021distributed}. There, a steepest descent algorithm on the sphere is analyzed using newly proved convexity-like properties of the spherical Rayleigh quotient.

Regarding the block version of the algorithm, where one targets multiple pairs of eigenvalues and eigenvectors, much less is known. We refer here to \cite{Neymer_block}, which presents a steepest descent-like method for the multiple eigenvector problem using Ritz projections onto a $2k$-dimensional subspace in each step. The convergence of this algorithm is proved to be linear, but computing the Ritz projections is quite expensive. Instead, in this work we consider a much cheaper version of steepest descent by directly choosing only one of the vectors in this $2k$-dimensional subspace to update our algorithm. Some analysis for such a steepest descent (without Ritz projection) on the Grassmann manifold using a retraction and an Armijo step-size is provided in \cite{absilOptimizationAlgorithmsMatrix2008} (see Algorithm 3 and Theorem 4.9.1). Unfortunately this convergence rate is asymptotic, that is, a linear rate is achieved after an unknown number of iterations, while the region of convergence cannot be quantified. Such a rate does not yield an iteration complexity.

The optimization landscape provided by the block Rayleigh quotient on the Grassmann manifold has also received some attention lately. \cite{Sato2014OptimizationAO}  provides many interesting properties of the critical points of this function and proves that all but the global optimum are strict saddles. This is later used to derive favourable convergence properties for a hybrid method consisting of Riemannian steepest descent in a first stage and a Riemannian Newton's method in a final stage. \cite{li2022landscape} proves the so-called robust strict saddle property for this function, that is, the Hessian evaluated in each critical point except the global optimum has both positive and negative eigenvalues in a whole neighborhood. However, none of these papers talks about (generalized) convexity of any form, nor discusses any convergence rates for steepest descent.

\hspace{4mm} Turning the discussion to the convexity properties of eigenvalue problems, there is a new line of research concerned by that. In \cite{zhang2016riemannian}, the authors prove (Theorem 4) that the Rayleigh quotient is geodesically gradient dominated in the sphere ($k=1$), that is, it satisfies a spherical version of the Polyak--Łojasiewicz inequality. In \cite{alimisis2021distributed}, it is shown that this result of \cite{zhang2016riemannian} can be strengthened to a geodesic weak-quasi-convexity and quadratic growth property, which imply gradient dominance when combined. Finally, the recent paper \cite{ahn2021riemannian} examines (among other contributions) the convexity structure of the same block version of the symmetric eigenvalue problem on the Grassmann manifold that we introduced above. Unfortunately, the characterization of the geodesic convexity region independently of the eigengap $\delta$ (Corollary 5 in \cite{ahn2021riemannian}) is wrong (see our Appendix~\ref{sec:convexity} for a counterexample). 
As we will prove in Theorem~\ref{eq:g-convex_domain}, the geodesic convexity region of $f$ (and the one of the equivalent cost function used in \cite{ahn2021riemannian}) needs to depend on the eigengap, as appears also in \cite[Lemma 7]{pmlr-v119-huang20e} in the case of the sphere ($k=1$).

To the best of our knowledge, the current work is the first that provides non-asymptotic convergence rates for the steepest descent algorithm for the multiple eigenvalue-eigenvector problem on the Grassmann manifold. We mainly rely on the work \cite{alimisis2021distributed}, which proves exponential convergence of steepest descent only in the case of $k=1$, that is, for the leading eigenvector. In this paper, we take a reasonable but highly non-trivial step forward by extending the convexity-like characterization of the spherical Rayleigh quotient to general $k$, that is, for a block of $k$ leading eigenvectors. Again, the paper \cite{bu2020note} is of high value for our current work regarding weakly-strongly-convex functions.

 As mentioned above, the standard algorithm for computing the leading eigenspace of dimension $k$ is subspace iteration (or power method when $k=1$).\footnote{Krylov methods are arguably the most popular algorithms but they do not iterate on a subspace directly and are typically started from a single vector. In particular, they cannot easily improve a given approximation of a subspace for large $k>1$.} However, there are reasons to believe that, in certain cases, Riemannian steepest descent (and its accelerated version with non-linear conjugate gradients) should be preferred, especially in noisy settings \cite{alimisis2021distributed} or in  electronic structure calculations where the leading eigenspace of many varying matrices $A$ needs to be computed.\footnote{Personal communication by Yousef Saad.} In particular, \cite{alimisis2021distributed} presents strong experimental evidence that steepest descent is more robust to perturbations of the matrix-vector products than subspace iteration close to the optimum. While subspace iteration still behaves better at the start of the iteration, it asymptotically fails to converge to an approximation of the leading subspace that is as good as the one estimated by Riemannian steepest descent. While \cite{alimisis2021distributed} dealt with a noisy situation due to calculations in a distributed setting with limited communication, exactly the same effect can be observed when we inject the matrix-vector products with Gaussian noise. Thus, we expect steepest descent to perform better than subspace iteration close to the optimum in any stochastic regime \cite{hardt2014noisy}.

Regarding worst-case theoretical guarantees, the strongest convergence result for subspace iteration in the presence of a strictly positive eigengap $\delta$ is in terms of the largest principal angle between the iterates and the optimum \cite{golub2013matrix}, that is, the $\ell_\infty$-norm of the vector of principal angles. In contrast, our convergence result for steepest descent for $\delta > 0$ (Theorem \ref{thm:exponential_conv}) is in terms of the $\ell_2$-norm of the same vector of angles, which is in general stronger. When $\delta = 0$, it is known from \cite{o1979estimating,Kuczynski92estimatingthe} that the largest eigenvalue ($k=1$) can still be efficiently estimated. We extend this result for $k>1$ and prove a convergence rate of steepest descent for the function values $f$ (Theorem \ref{thm:convex_conv}), relying only on weak-quasi-convexity (and thus using a different argument from \cite{o1979estimating,Kuczynski92estimatingthe}).

\section{Geometry of the Grassmann manifold and block Rayleigh quotient} \label{sec:background}

We present here a brief introduction into the geometry of the Grassmann manifold. The content is not new and for more details, we refer to \cite{absilOptimizationAlgorithmsMatrix2008,bendokatGrassmannManifoldHandbook2020,edelmanGeometryAlgorithmsOrthogonality1999}.

The $(n,k)$-Grassmann manifold is defined as the set of all $k$-dimensional subspaces of $\mathbb{R}^n$:
\begin{equation*}
    \Gr(n,k)=\lbrace \mathcal{X} \subseteq \mathbb{R}^n \colon \mathcal{X} \hspace{1mm}  \text{is a subspace and} \dim(\mathcal{X})=k \rbrace.
\end{equation*}

Any element $\mathcal{X}$ of $\Gr(n,k)$ can be represented by a matrix $X \in \mathbb{R}^{n \times k}$ that satisfies $\mathcal{X} = \myspan(X)$. Such a representative is not unique since $Y=XQ$ for some invertible matrix $Q \in  \mathbb{R}^{k \times k}$ satisfies $\myspan(Y) = \myspan(X)$. Without loss of generality, we will therefore always take matrix representatives $X$ of subspaces $\mathcal{X}$ that have orthonormal columns throughout the paper. 
With some care, the non-uniqueness of the representatives is not a problem.\footnote{This can be made very precise by describing $\Gr(n,k)$ as the quotient of the Stiefel manifold with the orthogonal group. The elegant theory of this quotient manifold is worked out in \cite{absilOptimizationAlgorithmsMatrix2008}.} For example, our objective function \eqref{eq:min_f_over_Gr} is invariant to $Q$.


\paragraph{Riemannian structure.} The set $\Gr(n,k)$ admits the structure of a differential manifold with tangent spaces
\begin{equation}\label{eq:def_TXGr}
    T_{\mathcal{X}} \Gr(n,k)=\lbrace G \in \mathbb{R}^{n \times k} \colon X^T G=0 \rbrace, 
\end{equation}
where $\mathcal{X} = \myspan(X)$. Since $X^T G = 0$ if and only if $(XQ)^T G=0$, for any invertible matrix $Q \in \mathbb{R}^{k \times k}$, this description of the tangent space does not depend on the representative $X$. However, a specific tangent vector $G$ will depend on the chosen $X$. With slight abuse of notation,\footnote{Using the quotient manifold theory, one would use horizontal lifts.} the above definition should therefore be interpreted as: given a fixed $X$, we define tangent vectors $G_1, G_2, \ldots $ of  $\Gr(n,k)$ at $\mathcal{X}=\myspan(X)$. 

This subtlety is important, for example, when defining an inner product on $T_{\mathcal{X}} \Gr(n,k)$:
\[
 \langle G_1, G_2 \rangle_{\mathcal{X}} = \Tr(G^T_1 G_2) \ \text{\ with\  $G_1,G_2 \in T_{\mathcal{X}} \Gr(n,k)$ }.
\] 
Here, $G_1$ and $G_2$ are tangent vectors of the same representative $X$. Observe that the inner product is invariant to the choice of orthonormal representative: If $\Bar{G}_1=G_1 Q$ and $\Bar{G}_2 = G_2 Q$ with orthogonal $Q$, then we have
\begin{equation*}
    \langle \bar G_1, \bar G_2 \rangle_{\mathcal{X}} = \Tr(\bar G^T_1  \bar G_2) = \Tr(Q^T G_1^T G_2 Q)= \Tr(G_1^T G_2 Q Q^T) = \Tr(G_1^T G_2).
\end{equation*}
It is easy to see that the norm induced by this inner product in any tangent space is the Frobenius norm, which we will denote throughout the paper as $\| \cdot \|:=\| \cdot \|_F$.

\paragraph{Exponential map.} Given the Riemannian structure of $\Gr(n,k)$, we can compute the exponential map at a point $\mathcal{X}$ as \cite[Thm.~3.6]{absilRiemannianGeometryGrassmann2004}
\begin{equation}\label{eq:formula_geo}
\begin{aligned}
 \Exp_{\mathcal{X}}: T_{\mathcal{X}} \Gr(n,k) &\rightarrow \Gr(n,k) \\ 
  G &\mapsto \myspan(\, X V \cos(\Sigma) + U \sin(\Sigma) \, ), 
\end{aligned}
\end{equation}
where $ U \Sigma V^T$ is the \emph{compact} SVD of $G$ such that $\Sigma$ and $V$ are square matrices. 

The exponential map is invertible in the domain \cite[Prop.~5.1]{bendokatGrassmannManifoldHandbook2020}
\begin{equation}\label{eq:inj_exp}
\left \lbrace G \in T_{\mathcal{X}} \Gr(n,k)  \colon  \| G \|_2 < \frac{\pi}{2} \right \rbrace,
\end{equation}
where $\| G \|_2$ is the spectral norm of $G$. The inverse of the exponential map restricted to this domain is the logarithmic map, denoted by $\Log$. Given two subspaces $\mathcal{X},\mathcal{Y}\in \Gr(n,k)$, we have
\begin{equation}\label{eq:log formula}
 \Log_{\mathcal{X}}(\mathcal{Y}) = U \atan(\widehat\Sigma) \, V^T,
\end{equation}
where  $U \widehat \Sigma V^T = (I - X X^T) Y (X^T Y)^{-1}$ is again a compact SVD. This is well-defined if $X^T Y$ is invertible, which is guaranteed if all principal angles between $\mathcal{X}$ and $\mathcal{Y}$ are strictly less than $\pi / 2$ (see below). By taking $G = \Log_{\mathcal{X}}(\mathcal{Y})$, we see that $\Sigma = \atan(\widehat\Sigma)$.

\paragraph{Principal angles.} The Riemannian structure of the Grassmann manifold can be conveniently described by the notion of the principal angles between subspaces. Given two subspaces $\mathcal{X},\mathcal{Y} \in \Gr(n,k)$, the principal angles between them are $0 \leq \theta_1 \leq \cdots \leq \theta_k \leq \pi/2$ obtained from the SVD 
\begin{equation}\label{eq:SVD_for_principal_angles}
    Y^T X=U_1 \cos \theta \ V_1^T
\end{equation}
where $U_1 \in \mathbb{R}^{k \times k}, V_1 \in \mathbb{R}^{k \times k}$ are orthogonal and the diagonal matrix $\cos \theta= \diag(\cos \theta_1,...,\cos \theta_k)$.

We can express the Riemannian logarithm using principal angles and the intrinsic distance induced by the Riemannian inner product discussed above is
\begin{equation}\label{eq:distance_with_Log_and_angles}
    \dist(\mathcal{X},\mathcal{Y)}=\| \Log_{\mathcal{X}} (\mathcal{Y}) \| = \| \Log_{\mathcal{Y}} (\mathcal{X}) \|=\sqrt{\theta_1^2+...+\theta_k^2}=\| \theta \|_2,
\end{equation}
where $\theta=(\theta_1, \ldots ,\theta_k)^T$.

If $X \in \mathbb{R}^{n \times k}$ is an arbitrary matrix with orthonormal columns, then, generically, these columns will not be exactly orthogonal to the $k$ leading eigenvectors $v_1, \ldots, v_k$ of $A$. Thus, we have with probability one that the principal angles between $\mathcal{X}$ and the space of $k$ leading eigenvectors satisfy $0 \leq \theta_1 \leq \cdots \leq \theta_k < \pi/2$.

\paragraph{Curvature.}
We can compute exactly the sectional curvatures in $\Gr(n,k)$, but for our purposes we only need that they are everywhere non-negative \cite{Wong,bendokatGrassmannManifoldHandbook2020}. This means that the geodesics on the Grassmann manifold spread more slowly than in Euclidean space. This is consequence of the famous Toponogov's theorem that we state here in the form of the following technical lemma, which will be important in our convergence analysis.
\begin{lemma}
\label{prop:tangent_space}
Let $\mathcal{X}, \mathcal{Y}, \mathcal{Z} \in \Gr(n,k)$, such that 
$$\max\{ \textnormal{dist}(\mathcal{X}, \mathcal{Z})  , \textnormal{dist} (\mathcal{Y}, \mathcal{Z}) \} < \frac{\pi}{2}.
$$ Then
\begin{equation*}
    \textnormal{dist}(\mathcal{X} , \mathcal{Y}) \leq \| \log_\mathcal{Z}(\mathcal{X})-\log_\mathcal{Z}(\mathcal{Y}) \|.
\end{equation*}

\end{lemma}

\begin{lemma}(Law of cosines) \label{lem:geo_triangle_nonneg}
Let $\mathcal{X},\mathcal{Y},\mathcal{Z}$ as in Lemma \ref{prop:tangent_space}. Then
\begin{equation*}
    \dist^2(\mathcal{X}, \mathcal{Y}) \leq \dist^2(\mathcal{Z}, \mathcal{X})+ \dist^2(\mathcal{Z}, \mathcal{Y})-2 \langle \Log_{\mathcal{Z}} (\mathcal{X}), \Log_{\mathcal{Z}} (\mathcal{Y}) \rangle.
\end{equation*}
\end{lemma}
\begin{proof}
Apply Lemma \ref{prop:tangent_space} and expand $\| \log_\mathcal{Z}(\mathcal{X})-\log_\mathcal{Z}(\mathcal{Y}) \|^2$.
\qed
\end{proof}

\paragraph{Block Rayleigh quotient.}
Our objective function for minimization is the block version of the Rayleigh quotient:
\[ 
f(\mathcal{X}) = -\Tr(X^TAX) \text{ where $\mathcal{X} = \myspan(X) \in \Gr(n,k)$ s.t. $X^T X = I_k$}.
\]  
This function has $\mathcal{V}_{\alpha}=\myspan(\begin{bmatrix} v_1 \hspace{2mm} \cdots \hspace{2mm} v_k \end{bmatrix})$ as global minimizer. This minimizer is unique on $\Gr(n,k)$ if and only if $\delta>0$.

Given any differentiable function $f\colon\Gr(n,k) \rightarrow \mathbb{R}$, we can define its Riemannian gradient as the vector field that satisfies
\begin{equation*}
    df(\mathcal{X})(\mathcal{G}) = \langle \grad f(\mathcal{X}),\mathcal{G} \rangle_{\mathcal{X}}.
\end{equation*}
For a given representative $X$ of $\mathcal{X}$, the Riemannian gradient of the block Rayleigh quotient satisfies 
\begin{equation*}\label{eq:grad f formula}
 \grad f(\mathcal{X}) = -2(I-X X^T) A X.
\end{equation*}
Using the notions of the Riemannian gradient and Levi-Civita connection, we can define also a Riemannian notion of Hessian. 
For the block Rayleigh quotient $f$, the Riemannian Hessian $\Hess f$ evaluated as bilinear form satisfies
\begin{equation}\label{eq:Hessian_f_inner_product}
  \Hess f(\mathcal{X})[{G},{G}] = 2 \langle G, G X^T A X - A G \rangle, 
\end{equation}
for  $G \in T_{\mathcal{X}} \Gr(n,k)$; see \cite[\S4.4]{edelmanGeometryAlgorithmsOrthogonality1999} or \cite[\S6.4.2]{absilOptimizationAlgorithmsMatrix2008}.

\section{Convexity-like properties of the block Rayleigh quotient} \label{sec:cost_conv}

We now prove the new analytic properties of the block Rayleigh quotient $f(\mathcal{X})=-\Tr(X^T A X)$. These are important in their own right but will also be used later for the convergence of the Riemannian steepest descent method.

\subsection{Smoothness}

A $C^2$ function defined on the Grassmann manifold is called $\gamma$-smooth if the maximum eigenvalue of its Riemannian Hessian is everywhere upper bounded by a positive constant $\gamma$. This is true for the block Rayleigh quotient, as we show in the next proposition:

\begin{tcolorbox} 
\begin{Proposition}[Smoothness]\label{prop:smoothness}
The eigenvalues of the Riemannian Hessian of $f$ on $\Gr(n,k)$ are upper bounded by  $\gamma := 2 (\lambda_1 - \lambda_n)$. 
\end{Proposition}
\end{tcolorbox}

\begin{proof}
Let $G$ be a tangent vector of $\Gr(n,k)$ at $X$. Then the Riemannian Hessian satisfies (see~\eqref{eq:Hessian_f_inner_product})
\[
  \tfrac{1}{2} \Hess f(\mathcal{X})[G,G] =  \Tr(G^T G X^T A X) - \Tr(A G G^T).
\]
Since $A, X^T A X, G G^T$, and $G^T G$ are all symmetric and positive semi-definite matrices, standard trace inequality (see, e.g, \cite[Thm.~4.3.53]{hornMatrixAnalysis2012a}) gives 
\[
 \Hess f(\mathcal{X})[G,G]  \leq 2 (\lambda_{\max}(X^T A X)- \lambda_{\min} (A)) \| G \|^2.
\]
Since $X$ has orthonormal columns, $\lambda_{\max}(X^T A X) \leq \lambda_{\max}( A )$; see, e.g., \cite[Cor.~4.3.37]{hornMatrixAnalysis2012a}. The proof is now complete with the definition of $\lambda_1$ and $\lambda_n$.
\qed
\end{proof}

The result in Prop.~\ref{prop:smoothness} is tight: Choosing $X = V_\alpha$ and $G = v_n e_1^T$, it is readily verified that the upper bound is attained. From now on, we refer to $\gamma$ as the specific value $2(\lambda_1-\lambda_n)$. This value also features in a useful upper bound for the spectral norm of the gradient. This bound is independent of $\mathcal{X}$:
\begin{tcolorbox}
\begin{lemma}\label{lem:uniform upper bound grad}
For all $\mathcal{X} \in \Gr(n,k)$, the Riemanian gradient of $f$ satisfies
\[
 \| \grad f(\mathcal{X}) \|_2  \leq \frac{\gamma}{2}.
\]
\end{lemma}
\end{tcolorbox}
\begin{proof}
Since $X$ has orthonormal columns, we can complete it to the orthogonal matrix $Q = \begin{bmatrix} X & X_\perp \end{bmatrix}$. Hence, $\| \grad f(\mathcal{X}) \|_2 = \| 2 (I - XX^T)AX \|_2 = 2 \| X_\perp^T A X \|_2$. The result now follows directly from~\cite[Thm.~2]{liInequalitiesSingularValues1999a} since $A$ is real symmetric and the definition of $\gamma = 2(\lambda_1 - \lambda_n)$.
\qed
\end{proof}

By the second-order Taylor expansion of $f$, it is easy to see that Proposition~\ref{prop:smoothness} implies (see, e.g., \cite[Thm.~7.1.2]{absilOptimizationAlgorithmsMatrix2008})
\begin{equation}\label{eq:quadratic_upper_bound}
    f(\mathcal{X}) \leq f(\mathcal{Y})+ \langle \grad f(\mathcal{Y}), \Log_{\mathcal{Y}} (\mathcal{X}) \rangle+\frac{\gamma}{2} \dist^2(\mathcal{X},\mathcal{Y}), 
\end{equation}
for any $\mathcal{X}, \mathcal{Y} \in \Gr(n,k)$ such that $\Log_{\mathcal{X}}(\mathcal{Y})$ is well-defined.

As in the introduction, denote the global minimum of $f$ by $f^*$ which is attained at $\mathcal{V}_\alpha \in \Gr(n,k)$. Inequality (\ref{eq:quadratic_upper_bound}) leads to the following lemma:

\begin{tcolorbox}

\begin{lemma}

For any $\mathcal{X} \in \Gr(n,k)$, we have
    \begin{equation*}\label{eq:optim_gap_with_gradient}
    f(\mathcal{X})-f^* \geq \frac{1}{2 \gamma} \| \grad f(\mathcal{X}) \|^2.
\end{equation*}

\end{lemma}
\end{tcolorbox}

\begin{proof}
Since $f^*$ is a global minimum of $f$, we have from~\eqref{eq:quadratic_upper_bound} that
\begin{equation*}
    f^* \leq f(\mathcal{X}) \leq f(\mathcal{Y})+\langle \textnormal{grad}f(\mathcal{Y}),\Log_{\mathcal{Y}}(\mathcal{X}) \rangle+\frac{\gamma}{2} \| \Log_{\mathcal{Y}}(\mathcal{X}) \|^2,
\end{equation*}
for any $\mathcal{X}, \mathcal{Y} \in \Gr(n,k)$ such that $\Log_{\mathcal{X}}(\mathcal{Y})$ is well-defined.

We set $\mathcal{X}:=\Exp_{\mathcal{Y}} \left(-\frac{1}{\gamma} \textnormal{grad}f(\mathcal{Y})\right)$. 
By Lemma \ref{lem:uniform upper bound grad}, we have that $\left\|-\frac{1}{\gamma} \textnormal{grad}f(\mathcal{Y}) \right \|_2<\frac{\pi}{2}$ and by equation \ref{eq:inj_exp} we have that $\Log_{\mathcal{Y}}(\mathcal{X})$ is well-defined and equal to $-\frac{1}{\gamma} \textnormal{grad}f(\mathcal{Y})$. Then,    
the right hand side of the initial inequality becomes 
\begin{align*}
    f^* \leq f(\mathcal{Y})-\frac{1}{\gamma} \| \textnormal{grad}f(\mathcal{Y})\|^2+ \frac{1}{2\gamma} \| \textnormal{grad}f(\mathcal{Y})\|^2 = f(\mathcal{Y})-\frac{1}{2 \gamma} \| \textnormal{grad}f(\mathcal{Y})\|^2,
\end{align*}
for any $\mathcal{Y} \in \Gr(n,k)$ with $\textnormal{dist}(\mathcal{Y},\mathcal{V}_{\alpha}) < \frac{\pi}{2}$. Rearranging the last inequality and substituting $\mathcal{Y}=\mathcal{X}$, we get the desired result.
\qed
\end{proof}

\subsection{Weak-quasi-convexity and quadratic growth}

We now turn our interest in the convexity properties of the block Rayleight quotient function. We start by proving a property which is known in the literature as \emph{quadratic growth}.

\begin{tcolorbox}
\begin{Proposition}[Quadratic growth]\label{prop:quadratic growth}
\label{prop:quadratic_growth}
Let $0 \leq \theta_1 \leq \cdots \leq \theta_k < \pi/2$ be the principal angles between the subspaces $\mathcal{X}$ and $\mathcal{V}_\alpha$. The function $f$ satisfies
$$f(\mathcal{X})-f^* \geq c_Q \, \delta \, \dist^2(\mathcal{X},\mathcal{V_{\alpha}})$$
where $c_Q = 4/\pi^2 > 0.4$.
\end{Proposition}
\end{tcolorbox}
\begin{proof}

The spectral decomposition of $A = V_{\alpha} \Lambda_{\alpha}V_{\alpha}^T + V_{\beta} \Lambda_{\beta} V_{\beta}^T$ implies
\begin{equation}\label{eq:XAX worked out}
    X^T AX = X^T V_{\alpha} \Lambda_{\alpha}V_{\alpha}^T X+ X^T V_{\beta} \Lambda_{\beta} V_{\beta}^T X.
\end{equation}
Since $f(\mathcal{X}) = -\Tr(X^TAX)$, we have
\begin{align*}
    f(\mathcal{X})-f^*  &=  \Tr(\Lambda_{\alpha})-\Tr(X^T V_{\alpha} \Lambda_{\alpha}V_{\alpha}^T X)-\Tr(X^T V_{\beta} \Lambda_{\beta} V_{\beta}^T X) \\ &= \Tr(\Lambda_{\alpha})-\Tr( \Lambda_{\alpha}V_{\alpha}^T X X^T V_{\alpha})-\Tr( \Lambda_{\beta} V_{\beta}^T X X^T V_{\beta}) \\ & = \Tr(\Lambda_{\alpha} (I_k- V_{\alpha}^T X X^T V_{\alpha})) - \Tr(\Lambda_{\beta} V_{\beta}^T X X^T V_{\beta}).
\end{align*}
From the definition~\eqref{eq:SVD_for_principal_angles} of the principal angles between $X$ and $V_{\alpha}$, we recall that 
\begin{equation}\label{eq:SVD_Va_X}
    V_{\alpha}^T X=U_1 \cos \theta \, V_1^T,
\end{equation}
where $\cos \theta = \diag(\cos \theta_1, \ldots, \cos \theta_k)$ is a diagonal matrix and $U_1, V_1$ are orthogonal matrices. Plugging this equality in, we get that the $j$th eigenvalue of the matrix $I_k- V_{\alpha}^T X X^T V_{\alpha}$ is equal to $1-\cos^2\theta_j = \sin^2 \theta_j \geq 0$. Thus, by standard trace inequality for symmetric and positive definite matrices (see, e.g.,~\cite[Thm.~4.3.53]{hornMatrixAnalysis2012a}), the first summand above satisfies
\begin{equation*}
    \Tr(\Lambda_{\alpha} (I_k- V_{\alpha}^T X X^T V_{\alpha})) \geq \lambda_{k} \sum_{j=1}^k \sin^2 \theta_j.
\end{equation*}
The matrix $V_{\beta}^T X X^T V_{\beta}$ has the same non-zero eigenvalues with the same multiplicity as the matrix 
\begin{equation*}\label{eq:SVD of Gramm of XVbeta}
       X^T V_{\beta} V_{\beta}^T X = I_k-V_1 \cos^2 \theta \, V_1^T = V_1 \sin^2 \theta \, V_1^T 
\end{equation*}
where we used $V_{\beta} V_{\beta}^T= I_n - V_{\alpha} V_{\alpha}^T$ and the SVD of $V_{\alpha}^T X$. 
Thus the $j$th eigenvalue of $V_{\beta}^T X X^T V_{\beta}$ is $\sin^2 \theta_j \geq 0$. By trace inequality again, the second summand therefore satisfies
\begin{equation*}
    \Tr(\Lambda_{\beta} V_{\beta}^T X X^T V_{\beta}) \leq \lambda_{k+1} \sum_{j=1}^k \sin^2 \theta_j.
\end{equation*}
Putting both bounds together, we get
\begin{equation*}
    f(\mathcal{X})-f^* \geq (\lambda_k - \lambda_{k+1}) \sum_{j=1}^k \sin^2 \theta_j
    \geq \delta \sum_{j=1}^k \frac{4}{\pi^2}\theta_j^2 
\end{equation*}
and the proof is complete by the definition~\eqref{eq:distance_with_Log_and_angles} of $\dist$.
\qed \end{proof}

We say that $f$ is geodesically convex if for all $\mathcal{X}$ and $\mathcal{Y}$ in a suitable region it holds
$$ f(\mathcal{X})-f(\mathcal{Y}) \leq \langle \textnormal{grad}f(\mathcal{X}), -\Log_{\mathcal{X}}(\mathcal{Y}) \rangle.$$
This generalizes the classical convexity of differentiable functions on $\mathbb{R}^n$ to manifolds by taking the logarithmic map instead of the difference $\mathcal{X}-\mathcal{Y}$.

In Appendix~\ref{sec:convexity}, we prove that our objective function $f$ is only geodesically convex in a small neighbourhood of size $\bigO(\sqrt{\delta})$ around the minimizer $\mathcal{V}_\alpha$. Fortunately, our key result of this section shows that $f$ satisfies a much weaker notion of geodesic convexity, known in the literature as \emph{weak-quasi-convexity}, that does not depend on the eigengap $\delta$.

We first need the following lemma which is the general version of the CS decomposition but applied to our setting of square blocks.
\begin{lemma}\label{lemma:CS_square_blocks}
Let $X,Y \in \mathbb{R}^{n \times k}$ be such that $X^T X = Y^T Y = I_k$ with $k < n$. 
Choose $X_\perp, Y_\perp \in \mathbb{R}^{n \times (n-k)}$ such that $X_\perp^T X_\perp = Y_\perp^T Y_\perp = I_{n-k}$ and $\myspan(X_\perp) = \myspan(X)^\perp$, $\myspan(Y_\perp) = \myspan(Y)^\perp$.
Then there exist $0 \leq r,s \leq k$ such that
\begin{align*}
 Y^T X &= U_1 \begin{bmatrix}I_r \\ & C_s \\ & & O_{p \times p} \end{bmatrix} V_1^T,  & 
 Y^T X_\perp &= U_1 \begin{bmatrix}O_{r \times m} \\ & S_s \\ & & I_{p} \end{bmatrix} V_2^T  \\
Y_\perp^T X &= U_2 \begin{bmatrix}O_{m \times r} \\ & S_s \\ & & I_{p} \end{bmatrix} V_1^T, & Y_\perp^T X_\perp &= U_2 \begin{bmatrix}-I_{m} \\ & -C_s \\ & & O_{p \times p} \end{bmatrix} V_2^T
\end{align*}
with $p=k-r-s$ and $m = n - 2k +r$, and we have
\begin{itemize}
\item orthogonal matrices $U_1, V_1$ of size $k$ and $U_2, V_2$ of size $n-k$;
\item identity matrices $I_q$ of size $q$;
\item zero matrices $O_{q \times t}$ of size $q \times t$;
\item diagonal matrices $C_s = \diag(\alpha_1, \ldots, \alpha_s)$ and $S_s = \diag(\beta_1, \ldots, \beta_s)$ such that $1 > \alpha_1 \ge \cdots \ge \alpha_s > 0$, $0 < \beta_1 \le \cdots \le \beta_s < 1$ and $C_s^2 + S_s^2 = I_s$.
\end{itemize}
\end{lemma}
\begin{proof}
Since $\begin{bmatrix} X & X_\perp \end{bmatrix}$ and $\begin{bmatrix} Y & Y_\perp \end{bmatrix}$ are orthogonal, the result follows directly from the CS decomposition of the orthogonal matrix $P = \begin{bmatrix} Y & Y_\perp \end{bmatrix}^T \begin{bmatrix} X & X_\perp \end{bmatrix}$; see the Theorem of \S 4 in \cite{paigeGeneralizedSingularValue1981}. 
\qed \end{proof}

Observe that the matrix $\diag(I_r,C_s,O_{p \times p})$  in this lemma corresponds to the matrix $\cos(\theta)$ in~\eqref{eq:SVD_for_principal_angles} with $\theta$ the vector of principal angles $0 \leq \theta_1 \leq \cdots \leq \theta_k \leq \pi/2$ between $\myspan(X)$ and $\myspan(Y)$. However, the lemma explicitly splits off the angles that are zero and $\pi/2$ so that it can formulate the related decompositions for $Y^T X_\perp, Y_\perp^T X,$ and $Y_\perp^T X_\perp$ with $C_s$ and $S_s$.

We are now ready to state our weak quasi-convexity result. In the statement of the proposition below (and throughout this paper), we use the convention that $\frac{0}{\tan 0} = 1$.

\begin{tcolorbox}
\begin{Proposition}[Weak-quasi-convexity]\label{prop:weak-quasi-convexity}
Let $0 \leq \theta_1 \leq \cdots \leq \theta_k < \pi/2$ be the principal angles between the subspaces  $\mathcal{X}$ and $\mathcal{V}_\alpha$. Then, $f$ satisfies
$$2 a(\mathcal{X}) \, (f(\mathcal{X})-f^*) \leq \langle \textnormal{grad}f(\mathcal{X}), -\Log_{\mathcal{X}}(\mathcal{V_{\alpha})} \rangle$$
with $a(\mathcal{X}) := \theta_k / \tan \theta_k$.
\end{Proposition}
\end{tcolorbox}


\begin{proof}
Take $X$ and $V_\alpha$ matrices with orthonormal columns such that $\mathcal{X} = \myspan(X)$ and $\mathcal{V}_\alpha = \myspan(V_\alpha)$. Since $\theta_k < \pi / 2$, we know that $p=0$ in Lemma~\ref{lemma:CS_square_blocks} and thus $s = k - r$ with $r$ the number of principal angles that are equal to zero. Choosing a matrix $X_\perp$ with orthonormal columns such that $\myspan(X_\perp) = \myspan(X)^\perp$, we therefore get from Lemma~\ref{lemma:CS_square_blocks} that there exist orthogonal matrices $U_1,V_1$ of size $k$ and $V_2$ of size $n-k$ such that
 \begin{align}\label{eq:SVD_XVa_XperpVa}
 V_\alpha^T X &= U_1 \begin{bmatrix}I_r \\ & C_{k-r}  \end{bmatrix} V_1^T,  & 
 V_\alpha^T X_\perp &= U_1 \begin{bmatrix}O_{r \times m} \\ & S_{k-r}  \end{bmatrix} V_2^T.
\end{align}
Comparing with~\eqref{eq:SVD_for_principal_angles}, we deduce that $C_{k-r} = \diag(\cos \theta_{r+1}  , \ldots, \cos \theta_k)$ and $S_{k-r} = \diag(\sin \theta_{r+1}, \ldots, \sin \theta_k)$ since $C_{k-r}^2 + S_{k-r}^2 = I$.

We recall from~\eqref{eq:log formula} that 
\begin{equation}\label{eq:log formula_XVa} 
 \Log_{\mathcal{X}}(\mathcal{V}_{\alpha}) = U \atan(\Sigma) V^T,
\end{equation}
where  $U \Sigma V^T=(I_n - X X^T) V_{\alpha} (X^T V_{\alpha})^{-1}=:M$ is a compact SVD (without the requirement that the diagonal of $\Sigma$ is non-increasing). Using $X_\perp$ from above, we can also write $M = X_\perp X_\perp^T V_{\alpha} (X^T V_{\alpha})^{-1}$. Substituting~\eqref{eq:SVD_XVa_XperpVa} and using that $U_1$ and $V_1$ are orthogonal gives 
\[
 M = X_\perp V_2 \begin{bmatrix}O_{m \times r\textbf{}} \\ & S_{k-r} C_{k-r}^{-1}  \end{bmatrix}   V_1^T 
 = X_\perp \tilde V_2 \begin{bmatrix}O_{r \times r} \\ & S_{k-r} C_{k-r}^{-1}  \end{bmatrix}   V_1^T,
\]
where $\tilde V_2 \in \mathbb{R}^{(n-k) \times k}$ contains the last $k$ columns of $V_2$ in order. Note that this reformulation of the SVD of $M$ holds always, regardless of the relationship between $m$ and $r$. If $m \geq r$, the matrix $\begin{bmatrix}O_{m \times r} \\ & S_{k-r} C_{k-r}^{-1}  \end{bmatrix}$ has its first $m-r$ rows equal to $0$, thus we can cut the first $m-r$ columns of $V_2$, since they do not contribute to the product. This yields a matrix $\tilde V_2$ with $n-k$ rows and $n-k-m+r=k$ of the last columns of $V_2$. If $m<r$, then the first $r-m$ columns of $\begin{bmatrix}O_{m \times r\textbf{}} \\ & S_{k-r} C_{k-r}^{-1}  \end{bmatrix}$ are $0$ and now we can add $r-m$ columns in the beginning of the matrix $V_2$ that keep the derived 
 matrix orthonormal. This again yields a matrix $\Tilde V_2$ with $n-k$ rows and $n-k+r-m=k$ columns. Since the matrix $\begin{bmatrix}O_{r \times r\textbf{}} \\ & S_{k-r} C_{k-r}^{-1}  \end{bmatrix}$ occurs by adding $r-m$ zero rows at the beginning of $\begin{bmatrix}O_{m \times r} \\ & S_{k-r} C_{k-r}^{-1}  \end{bmatrix}$, the product does not change.

Since $\theta_1 = \cdots = \theta_r = 0$, we can therefore formulate the compact SVD of $M$ using the vector $\theta$ of all principal angles as follows:
\[
 M = U \Sigma V^T \quad \text{with $U = X_\perp \tilde V_2, \ \Sigma = \tan(\theta), \ V = V_1$}.
\]
Hence from~\eqref{eq:log formula_XVa} we get directly that
\begin{equation}\label{eq:Log_with_V2}
 \Log_{\mathcal{X}}(\mathcal{V}_{\alpha}) =  X_\perp \tilde V_2\, \theta \, V_1^T,
\end{equation}
where $\theta$ is a diagonal matrix.

We now claim that~\eqref{eq:Log_with_V2} also satisfies
\begin{equation}\label{eq:Log_with_Va}
 \Log_{\mathcal{X}}(\mathcal{V}_{\alpha}) =  X_\perp X_\perp^T V_\alpha U_1 \frac{\theta}{\sin \theta} V_1^T,
\end{equation}
where $\frac{\theta}{\sin \theta}$ is a diagonal matrix for which $\frac{0}{\sin 0} = 1$. Indeed, recalling that $\theta_1 = \cdots = \theta_r = 0$ and using the identities
\[
X_\perp^T V_\alpha =  \tilde V_2 \begin{bmatrix}O_{r \times r} \\ & S_{k-r}  \end{bmatrix} U_1^T, \quad\frac{\theta}{\sin \theta} = \begin{bmatrix}I_{r} \\ & S_{k-r}^{-1}  \end{bmatrix} \begin{bmatrix}I_{r} \\ & T_{k-r} \end{bmatrix}
\]
where $T_{k-r}=\diag(\theta_{r+1}, \ldots, \theta_k)$, we obtain
\begin{align*}
\textrm{rhs of~\eqref{eq:Log_with_Va}} &= X_\perp \tilde V_2 \begin{bmatrix}O_{r \times r} \\ & S_{k-r}  \end{bmatrix} \begin{bmatrix}I_{r} \\ & S_{k-r}^{-1}  \end{bmatrix} \begin{bmatrix}I_{r} \\ &  T_{k-r} \end{bmatrix} \, V_1^T \\ & =  X_\perp \tilde V_2 \begin{bmatrix} O_{r \times r} \\ & T_{k-r}  \end{bmatrix} \, V_1^T   = X_\perp \tilde V_2 \, \theta \, V_1^T = \textrm{rhs of~\eqref{eq:Log_with_V2}}.
\end{align*}



Next, we work out 
\[
 s := \langle \grad f(\mathcal{X}), -\Log_{\mathcal{X}}(\mathcal{V_{\alpha})} \rangle.
\] 
Since $\grad f(\mathcal{X})$ and $\Log_{\mathcal{X}}(\mathcal{V}_{\alpha})$, respectively, give tangent vectors for the same representative $X$ of $\mathcal{X}$, the inner product above is the trace of the corresponding matrix representations. Using~\eqref{eq:Log_with_Va} with $I-XX^T = X_\perp X_\perp^T$, we therefore get
\begin{align*}
 s &= 2 \Big\langle (I-XX^T)AX, (I-XX^T) V_\alpha U_1 \frac{\theta}{\sin(\theta)} V_1^T \Big\rangle \\
  &= 2 \Tr \Big( \frac{\theta}{\sin(\theta)} U_1^T V_\alpha^T (I-XX^T) AX V_1 \Big). 
\end{align*}
Since $AV_\alpha = V_\alpha \Lambda_\alpha$, we can simplify
\begin{equation}\label{eq:grad_with_Va}
 V_\alpha^T (I-XX^T) AX = \Lambda_\alpha V_\alpha^T X - V_\alpha^T XX^T AX.
\end{equation}
Substituting in the expression above and using that $V_{\alpha}^T X=U_1 \cos \theta \, V_1^T$, we get
\begin{align*}
 \frac{1}{2} s &= \Tr \Big( \frac{\theta}{\sin(\theta)} U_1^T  \Lambda_\alpha U_1 \cos(\theta) \Big) - 
 \Tr \Big( \frac{\theta}{\sin(\theta)} \cos(\theta) V_1^T  X^T AX  V_1 \Big)  \\
 &= \Tr \Big( \frac{\theta}{\tan(\theta)} \Big( U_1^T  \Lambda_\alpha U_1 - V_1^T  X^T AX  V_1 \Big) \Big),
\end{align*}
with the convention $\frac{0}{\tan 0} = 1$.

Denote the symmetric matrix
\begin{equation}\label{eq:def_S}
 S := U_1^T  \Lambda_\alpha U_1 - V_1^T  X^T AX  V_1.
\end{equation}
We show below that all diagonal entries $S_{11}, \ldots, S_{kk}$ of $S$ are nonnegative. Hence, by diagonality of the matrix $\tfrac{\theta}{\tan(\theta)}$, we obtain
\textbf{}\begin{align*}
 \frac{1}{2} s &=  \sum_j \frac{\theta_j}{\tan\theta_j} \, S_{jj} \geq \min_j \frac{\theta_j}{\tan\theta_j} \, \Tr(S) = \frac{\theta_k}{\tan\theta_k} \, \Big[ \Tr( \Lambda_{\alpha}) -  \Tr( X^T AX ) \Big]
\end{align*}
since $U_1$ and $V_1$ are orthogonal matrices. We recover the desired result after substituting $f(\mathcal{X}) = -\Tr(X^T A X)$ and $f_* = -\Tr(V_\alpha^T A V_\alpha) = -\Tr(\Lambda_\alpha)$.

It remains to show that $S_{jj} \geq 0$ for $j=1,\ldots, k$. Since $\myspan(V_\beta) = \myspan(V_\alpha)^\perp$, Lemma~\ref{lemma:CS_square_blocks} gives us in addition to~\eqref{eq:SVD_XVa_XperpVa} also
\begin{equation}\label{eq:SVD_Vb_X}
 V_\beta^T X = U_2 \begin{bmatrix} O_{m \times r} \\ & S_{\textcolor{yellow}{k-r}} \end{bmatrix} V_1^T = \tilde U_2 \sin \theta \,  V_1^T,
\end{equation}
where $\tilde U_2 \in \mathbb{R}^{(n-k) \times k}$ contains the last $k$ columns of the orthogonal matrix $U_2$ in order. 
A short calculation using~\eqref{eq:XAX worked out} then shows that~\eqref{eq:def_S} satisfies
\[
 S= U_1^T  \Lambda_\alpha U_1  - \cos\theta \, U_1^T  \Lambda_\alpha U_1  \cos\theta - \sin\theta \, \tilde U_2^T  \Lambda_\beta\tilde  U_2  \sin\theta
\]
with diagonal elements 
\[
 S_{jj} 
 = \sin^2\theta_j \, (U_1^T  \Lambda_\alpha U_1 -\tilde  U_2^T  \Lambda_\beta\tilde  U_2)_{jj}.
\]
Since $U_1$ and $\tilde U_2$ have orthonormal columns, we obtain
\[
 \lambda_{\min} (U_1^T  \Lambda_\alpha U_1) \geq \lambda_{\min} (\Lambda_\alpha) = \lambda_k, \quad \lambda_{\max}(\tilde U_2^T  \Lambda_\beta \tilde U_2) \leq \lambda_{\max}(\Lambda_\beta) = \lambda_{k+1},
\]
from which we get with Weyl's inequality that
\[
 \lambda_{\min} (U_1^T  \Lambda_\alpha U_1 - \tilde U_2^T  \Lambda_\beta \tilde U_2)
 \geq \lambda_{\min} (U_1^T  \Lambda_\alpha U_1) - \lambda_{\max}(\tilde U_2^T  \Lambda_\beta \tilde U_2) \geq \lambda_k - \lambda_{k+1} \geq 0.
\]
Hence, the matrix 
\begin{equation}\label{eq:La_min_Lb_is_PSD}
U_1^T  \Lambda_\alpha U_1 - \tilde U_2^T  \Lambda_\beta \tilde U_2
\end{equation}
is symmetric and positive semi-definite. Its diagonal entries, and thus also $S_{jj}$, are therefore nonnegative. 
\qed \end{proof}

We now arrive at a useful property of $f$ that will later allow us to analyze the convergence of Riemannian steepest descent. It is a \emph{weaker version of strong geodesic convexity} and can be proved easily using quadratic growth and weak-quasi-convexity. 

\begin{tcolorbox}
\begin{theorem}[Weak-strong convexity]\label{thm:weak_strong_convex}
Let $0 \leq \theta_1 \leq \cdots \leq \theta_k < \pi/2$ be the principal angles between the subspaces $\mathcal{X}$ and $\mathcal{V}_\alpha$. Then, $f$ satisfies
\begin{equation*}
   f(\mathcal{X})-f^* \leq \frac{1}{a(\mathcal{X})} \langle \grad f(\mathcal{X}), -\Log_{\mathcal{X}}(\mathcal{V}_{\alpha}) \rangle - c_Q \delta \, \dist^2 (\mathcal{X},\mathcal{V}_{\alpha})
\end{equation*}
with $a(\mathcal{X})= \theta_k / \tan \theta_k >0$, $c_Q = 4/\pi^2 > 0.4$, and $\delta = \lambda_k - \lambda_{k+1} \geq 0$.
\end{theorem}
\end{tcolorbox}

\begin{proof}
Combining Propositions~\ref{prop:quadratic_growth} and~\ref{prop:weak-quasi-convexity} leads to
\begin{equation*}
    c_Q \delta \, \dist^2(\mathcal{X},\mathcal{V}_{\alpha}) \leq f(\mathcal{X})-f^* \leq \frac{1}{2 a(\mathcal{X})} \langle \grad f(\mathcal{X}), -\Log_{\mathcal{X}}(\mathcal{V}_{\alpha}) \rangle.
\end{equation*}
At the same time, Proposition~\ref{prop:weak-quasi-convexity} also implies
\begin{align*}
    f(\mathcal{X})-f^*  &\leq \frac{1}{2 a(\mathcal{X})} \langle \grad f(\mathcal{X}), -\Log_{\mathcal{X}}(\mathcal{V}_\alpha) \rangle - c_Q \delta \, \dist^2(\mathcal{X},\mathcal{V}_\alpha) \\ & \qquad + c_Q \delta \, \dist^2(\mathcal{X},\mathcal{V}_\alpha).
\end{align*}
Substituting this inequality in the first one gives the desired result.
\qed \end{proof}
\begin{remark}
Theorem \ref{thm:weak_strong_convex} is also valid when the eigengap $\delta = 0$. In that case, $\mathcal{V}_{\alpha}$ is \textit{any subspace spanned by $k$ leading eigenvectors of $A$} and the theorem reduces to Proposition \ref{prop:weak-quasi-convexity}.
\end{remark}


While not needed for our convergence proof, the next result is of independent interest and shows that $f$ is gradient dominated in the Riemannian sense when the eigengap $\delta$ is strictly positive. This property is the \emph{Riemannian version of the Polyak--Łojasiewicz inequality} and generalizes a result by \cite{zhang2016riemannian} for the Rayleigh quotient on the sphere.

\begin{tcolorbox}
\begin{Proposition}[Gradient dominance] \label{prop:PL}
The function $f$ satisfies
\begin{equation*}
    \| \textnormal{grad}f(\mathcal{X}) \|^2 \geq 4 \, c_Q \, \delta \, a^2(\mathcal{X}) (f(\mathcal{X})-f^*)
\end{equation*}
for all subspaces $\mathcal{X}$ that have a largest principal angle $<\pi/2$ with $\mathcal{V}_\alpha$.
\end{Proposition}
\end{tcolorbox}

\begin{proof}
We assume that $\delta > 0$ since otherwise the statement is trivially true. By Theorem~\ref{thm:weak_strong_convex}, we have
\begin{equation*}
   f(\mathcal{X})-f^* \leq \frac{1}{a(\mathcal{X})} \langle \grad f(\mathcal{X}), -\Log_{\mathcal{X}}(\mathcal{V}_{\alpha}) \rangle - c_Q \delta \dist^2 (\mathcal{X}, \mathcal{V}_{\alpha}).
\end{equation*}
Since $\langle G_1, G_2 \rangle \leq (\|G_1\|^2 + \|G_2 \|^2)/2$ for all matrices $G_1, G_2$, we can write for any $\rho > 0$ that
\begin{equation*}
    \langle \grad f(\mathcal{X}), -\Log_{\mathcal{X}}(\mathcal{V}_{\alpha}) \rangle \leq \frac{\rho}{2} \| \textnormal{grad}f(\mathcal{X}) \|^2 + \frac{1}{2 \rho} \| \Log_{\mathcal{X}}(\mathcal{V}_{\alpha}) \|^2.
\end{equation*}
Using that $\dist (\mathcal{X}, \mathcal{V}_{\alpha})= \| \Log_{\mathcal{X}}(\mathcal{V}_{\alpha}) \|$ and choosing $\rho=1/(2 c_Q \delta a(\mathcal{X}))$, we get the desired result.
\qed \end{proof}

\section{Convergence of Riemannian steepest descent}

We now have everything in place to prove the convergence of the Riemannian steepest descent (RSD) method on the Grassmann manifold for minimizing $f$. Starting from a subspace $\mathcal{X}_0 \in \Gr(n,k)$, we iterate
\begin{equation}\label{eq:GD}
 \mathcal{X}_{t+1}=\Exp_{\mathcal{X}_t} (-\eta_t \, \grad f(\mathcal{X}_t) ).
\end{equation}
Here, $\eta_t > 0$ is a step size that may depend on the iteration $t$ and will be carefully chosen depending on the specific case.


We start by a general result which shows that the distance to the optimal subspace contracts after one step of steepest descent. The step size depends on the smoothness and weak-quasi-convexity  constants of $f$ from Propositions~\ref{prop:smoothness} and~\ref{prop:weak-quasi-convexity}. This is crucial since the constant $a(\mathcal{X})$ depends on the biggest principal angle between $\mathcal{X}$ and $\mathcal{V}_{\alpha}$ and bounding the evolution of distances of the iterates to the minimizer will help us also bound this constant.\footnote{The analysis of \cite{pmlr-v119-huang20e} is wrong with respect to this issue as discussed in detail in \cite{alimisis2021distributed}.} An alternative contraction property with a more tractable step size is presented in Proposition \ref{prop:big_step_distance} of Appendix \ref{sec:big_step}. 
\begin{tcolorbox}
\begin{lemma}[Contraction of RSD]\label{lem:GD convergence 1 step}
Let $\mathcal{X}_t$ and $\mathcal{V}_\alpha$ have principal angles $0 \le \theta_1 \leq \cdots \leq \theta_k < \pi/2$. 
Then, iteration~\eqref{eq:GD} with $0 \le \eta_t \leq \frac{a(\mathcal{X}_t)}{\gamma}$ satisfies
   \begin{equation*}
   \dist^2(\mathcal{X}_{t+1},\mathcal{V}_{\alpha})   \leq \big(1- 2 c_Q \delta a(\mathcal{X}_t) \, \eta_t \big)  \dist^2(\mathcal{X}_t,\mathcal{V}_{\alpha}) 
   \end{equation*}
\end{lemma}
\end{tcolorbox}

Observe that $\gamma= 0$ implies $A = \lambda_1 I$ and any subspace $\mathcal{X}$ of dimension $k$ will be an eigenspace of $A$ with $\dist(\mathcal{X},\mathcal{V}_\alpha)=0$. We will therefore not explicitly prove this lemma and all forthcoming convergence results for $\gamma=0$ since the statements will be trivially true.


\begin{proof}[Proof of Lemma~\ref{lem:GD convergence 1 step}]
By the assumption on the principal angles, we get that $0< a(\mathcal{X}_t) = \theta_k / \tan \theta_k \leq 1$. 
The hypothesis on $\eta_t$ and Lemma~\ref{lem:uniform upper bound grad} then gives
\[
 \eta_t \|\grad f(\mathcal{X}_t)\|_2 \leq \frac{a(\mathcal{X}_t)}{\gamma} \|\grad f(\mathcal{X}_t)\|_2 \leq \frac{1}{2} < \frac{\pi}{2}.
\]
By~\eqref{eq:inj_exp}, this guarantees that the geodesic $\tau \mapsto \Exp(- \tau \eta_t \, \grad f(\mathcal{X}_t))$ lies within the injectivity domain at ${\mathcal{X}_t}$ for $\tau \in [0,1]$. Hence, $\Exp$ is bijective along this geodesic and thus $\Log_{\mathcal{X}_t}(\mathcal{X}_{t+1}) = -\eta_t \, \grad f(\mathcal{X}_t)$. We can thus apply Lemma~\ref{prop:tangent_space} to obtain
\begin{align}\label{eq:dist_with_sigma}
        \dist^2(\mathcal{X}_{t+1},\mathcal{V}_{\alpha})  &\leq \|-\eta_t \grad f(\mathcal{X}_t)-\textnormal{Log}_{\mathcal{X}_t}(\mathcal{V}_{\alpha}) \|^2   \notag \\ 
        & = \eta_t^2 \|\grad f(\mathcal{X}_t)\|^2 + \dist^2(\mathcal{X}_t, \mathcal{V}_{\alpha}) +2 \eta_t \, \sigma \
    \end{align}
    with
    \[
    \sigma := \langle \grad f(\mathcal{X}_t),\Log_{\mathcal{X}_t}(\mathcal{V}_{\alpha}) \rangle.
    \]
    Theorem~\ref{thm:weak_strong_convex} and~\eqref{eq:optim_gap_with_gradient} together with Proposition~\ref{prop:smoothness} give
    \begin{align*}
        \frac{\sigma}{a(\mathcal{X}_t)} &\leq f^*-f(\mathcal{X}_t)-c_Q \delta \dist^2(\mathcal{X}_t,\mathcal{V}_{\alpha}) \\
        &\leq -\frac{1}{2 \gamma} \| \grad f(\mathcal{X}_t) \|^2- c_Q \delta \dist^2(\mathcal{X}_t,\mathcal{V}_{\alpha}).
    \end{align*}
    Multiplying by $2  a(\mathcal{X}_t)\, \eta_t$ and using $\eta_t \leq a(\mathcal{X}_t) / \gamma$, we get
    \begin{align*}
        2 \eta_t \, \sigma &\leq -\frac{a(\mathcal{X}_t) \, \eta_t }{\gamma} \| \grad f(\mathcal{X}_t) \|^2 - 2 c_Q \delta a(\mathcal{X}_t) \, \eta_t \, \dist^2(\mathcal{X}_t,\mathcal{V}_{\alpha}) \\ & \leq -\eta_t^2 \| \grad f(\mathcal{X}_t) \|^2 - 2 c_Q \delta a(\mathcal{X}_t)\, \eta_t \, \dist^2(\mathcal{X}_t, \mathcal{V}_{\alpha}).
         \end{align*}
         Substituting into~\eqref{eq:dist_with_sigma}, we obtain the first statement of the lemma.
\qed 
\end{proof}

\begin{remark}
When $\delta=0$, Lemma \ref{lem:GD convergence 1 step} still holds \emph{for any subspace $\mathcal{V}_{\alpha}$ spanned by $k$ leading eigenvectors of $A$}. In that case, the lemma only guarantees that the distance between the iterates of steepest descent and this $\mathcal{V}_{\alpha}$ does not increase. 
\end{remark}

\subsection{Linear convergence rate under positive eigengap}
When $\delta>0$, we can extend Lemma \ref{lem:GD convergence 1 step} to a linear convergence rate of distances to the minimizer:
\begin{tcolorbox}
\begin{theorem} \label{thm:exponential_conv}
If $\dist(\mathcal{X}_0,\mathcal{V}_{\alpha}) < \pi / 2$ then the iterates $\mathcal{X}_t$ of Riemannian steepest descent~\eqref{eq:GD} with step size $\eta_t$ such that
\begin{equation*}
    0<\eta \leq \eta_t \leq \cos(\dist(\mathcal{X}_0,\mathcal{V}_{\alpha})) / \gamma
\end{equation*}
 satisfy
\begin{equation*}
  \textnormal{dist}^2(\mathcal{X}_t,\mathcal{V}_{\alpha})   \leq \left(1- 2 c_Q \cos (\dist(\mathcal{X}_0,\mathcal{V}_{\alpha}))\, \delta\, \eta \right) ^ t  \dist^2(\mathcal{X}_0,\mathcal{V}_{\alpha}).  
\end{equation*}
\end{theorem}
\end{tcolorbox}

\begin{proof}
We first claim that $\dist (\mathcal{X}_{t},\mathcal{V}_{\alpha}) \leq \dist (\mathcal{X}_0,\mathcal{V}_{\alpha})$ for all $t \geq 0$. This would then also imply that $\theta_k(\mathcal{X}_t, \mathcal{V}_\alpha) < \pi/2$ for all  $t\geq 0$ since
\[
 \theta_k(\mathcal{X}_t, \mathcal{V}_\alpha) \leq \sqrt{\sum_{i=1}^k \theta_i (\mathcal{X}_t, \mathcal{V}_\alpha)^2} = \dist(\mathcal{X}_t, \mathcal{V}_\alpha).
\]
For $t=0$, we have $\theta_k(\mathcal{X}_{0},\mathcal{V}_{\alpha}) < \pi/2$ by hypothesis on $\mathcal{X}_0$ and thus
\begin{equation*}
    a(\mathcal{X}_0)=\frac{\theta_k(\mathcal{X}_0,\mathcal{V}_{\alpha})}{\tan(\theta_k(\mathcal{X}_0,\mathcal{V}_{\alpha}))} \geq \cos(\theta_k(\mathcal{X}_0,\mathcal{V}_{\alpha})) \geq \cos(\dist(\mathcal{X}_0,\mathcal{V}_{\alpha})).
\end{equation*}
Since by construction $\eta_0 \leq \cos(\dist(\mathcal{X}_0,\mathcal{V}_{\alpha})) / \gamma$ , this implies that $\eta_0 \leq a(\mathcal{X}_0) / \gamma$ and Lemma \ref{lem:GD convergence 1 step} guarantees that $\dist(\mathcal{X}_{1},\mathcal{V}_{\alpha}) \leq \textnormal{dist}(\mathcal{X}_0,\mathcal{V}_{\alpha})$. In particular, we also have $\theta_k(\mathcal{X}_{1},\mathcal{V}_{\alpha}) < \pi/2$.

Next, assume that 
\begin{equation*}
    \dist (\mathcal{X}_t,\mathcal{V}_{\alpha}) \leq \dist (\mathcal{X}_{0},\mathcal{V}_{\alpha}),
\end{equation*}
which implies $\theta_k(\mathcal{X}_{t},\mathcal{V}_{\alpha}) < \pi/2$.
Then by a similar argument like above, we have
\begin{equation}\label{eq:aX_with_cos_X0}
   a(\mathcal{X}_t) 
   \geq \cos(\dist(\mathcal{X}_t,\mathcal{V}_{\alpha})) \geq \cos(\dist(\mathcal{X}_{0},\mathcal{V}_{\alpha})).
\end{equation}
By hypothesis on $\eta_t$, we observe
\begin{equation*}
    \eta_t \leq \frac{\cos(\dist(\mathcal{X}_0,\mathcal{V}_{\alpha}))}{\gamma} \leq \frac{\cos(\dist(\mathcal{X}_t,\mathcal{V}_{\alpha}))}{\gamma}
    \leq \frac{a(\mathcal{X}_t)}{\gamma}.
\end{equation*}
Applying Lemma~\ref{lem:GD convergence 1 step} once again with the induction hypothesis proves the claim:
\begin{equation*}
    \dist (\mathcal{X}_{t+1},\mathcal{V}_{\alpha}) \leq \dist (\mathcal{X}_t,\mathcal{V}_{\alpha}) \leq \dist (\mathcal{X}_0,\mathcal{V}_{\alpha}).
\end{equation*}

The main statement of the theorem now follows easily: Since $\eta_t \leq a(\mathcal{X}_t) / \gamma$ and  $\theta_k(\mathcal{X}_{t}, \mathcal{V}_{\alpha}) < \pi/2$ for all $t\geq 0$, Lemma \ref{lem:GD convergence 1 step} gives
\begin{equation*}
    \textnormal{dist}^2(\mathcal{X}_{t+1},\mathcal{V}_{\alpha})   \leq \left(1- 2 c_Q a(\mathcal{X}_t) \delta \eta_t \right) \textnormal{dist}^2(\mathcal{X}_t,\mathcal{V}_{\alpha}).
\end{equation*}
Combining with~\eqref{eq:aX_with_cos_X0} and $\eta_t \geq \eta$ shows the desired result by induction.
\qed \end{proof}

If the eigengap $\delta$ is strictly positive, then Theorem \ref{thm:exponential_conv} gives an exponential convergence rate towards the optimum $\mathcal{V}_{\alpha}$. If $\delta=0$, then Theorem \ref{thm:exponential_conv} \emph{does not provide a convergence rate} but rather implies that the intrinsic distances of the iterates to the optimum do not increase.

From Theorem \ref{thm:exponential_conv} we get immediately the following iteration complexity.

\begin{corollary}
Let Riemannian steepest descent be started from a subspace $\mathcal{X}_0$ that satisfies $\dist(\mathcal{X}_0,\mathcal{V}_{\alpha}) < \pi/2$. Then after at most
\begin{equation*}
    T \leq 2 \frac{ \log(\varepsilon) - \log(\dist(\mathcal{X}_0,\mathcal{V}_{\alpha}))}{\log(1- 0.8 \cos (\dist(\mathcal{X}_0,\mathcal{V}_{\alpha})) \delta \eta)}  +1 
        = \bigO \left(\frac{\log(\dist(\mathcal{X}_0,\mathcal{V}_{\alpha})) - \log(\varepsilon)}{\cos (\dist(\mathcal{X}_0,\mathcal{V}_{\alpha})) \delta \eta}  \right)
\end{equation*}
    many iterations, $\mathcal{X}_T$ will satisfy $\dist(\mathcal{X}_T,\mathcal{V}_{\alpha}) \leq \varepsilon$. With the maximal step size allowed in Theorem \ref{thm:exponential_conv}, we get
\[
 T = \bigO \left(\frac{\lambda_1 - \lambda_n}{\delta}
 \frac{1}{\cos^2(\dist(\mathcal{X}_0,\mathcal{V}_{\alpha}))}
 \log\left(\frac{\dist(\mathcal{X}_0,\mathcal{V}_{\alpha})}{\varepsilon}\right) \right).
\]
\end{corollary}

As expected, $T$ depends inversely proportional on the eigengap $\delta$ and proportional to the spread of the eigenvalues. In addition, we also have an extra term $1/\cos^2(\dist(\mathcal{X}_0,\mathcal{V}_{\alpha}))$ that depends on the initial distance $\dist(\mathcal{X}_0,\mathcal{V}_{\alpha})$, which is due to the weak-quasi-convexity property of $f$. This is a conservative overestimation, since this quantity improves as the iterates get closer to the optimum. 


\begin{remark}
  If $\delta>0$, the exponential convergence rate is in terms of the intrinsic distance on the Grassmann manifold, that is, the $\ell_2$ norm of the principal angles. Standard convergence results for subspace iteration are stated for the biggest principal angle, that is, the $\ell_\infty$ norm. This is weaker than the intrinsic distance. For subspace iteration with projection, the convergence result from~\cite[Thm.~5.2]{saadNumericalMethodsLarge2011} shows that all principal angles $\theta_i$ converge to zero and eventually gives convergence of the $\ell_4$ norm of the principal angles. This is also weaker than the intrinsic distance.
\end{remark}



\subsection{Convergence of function values without an eigengap assumption}
When $\delta=0$, Theorem \ref{thm:exponential_conv} still holds, but does not provide a rate of convergence as discussed above. Instead, we can prove the following result:

\begin{tcolorbox}
\begin{theorem} \label{thm:convex_conv}
If the distance $\dist(\mathcal{X}_0,\mathcal{V}_{\alpha})$ of the initial subspace $\mathcal{X}_0$ to the minimizer satisfies $\dist(\mathcal{X}_0,\mathcal{V}_{\alpha})<\pi / 2$ for a subspace $\mathcal{V}_{\alpha}$ that is spanned by any $k$ leading eigenvectors of $A$, then the iterates $\mathcal{X}_t$ of Riemannian steepest descent~\eqref{eq:GD} with fixed step size 
\begin{equation*}
    \eta \leq \cos(\dist(\mathcal{X}_0,\mathcal{V}_{\alpha})) / \gamma
\end{equation*}
 satisfy
\begin{equation*}
    f(\mathcal{X}_t) - f^* \leq \frac{2\gamma+\frac{1}{\eta}}{4} (\cos(\dist(\mathcal{X}_0,\mathcal{V}_{\alpha}))t+1) \dist^2(\mathcal{X}_0,\mathcal{V}_{\alpha})=\bigO\left(\frac{1}{t} \right).  
\end{equation*}
\end{theorem}
\end{tcolorbox}


\begin{proof}
Since we satisfy all the hypotheses of Theorem \ref{thm:exponential_conv}, we know that for all $t\geq 0$ it holds
$\textnormal{dist}(\mathcal{X}_{t},\mathcal{V}_{\alpha}) \leq 
\textnormal{dist}(\mathcal{X}_0,\mathcal{V}_{\alpha}) < \pi/2$ and thus also that $\mathcal{X}_t$ is in the injectivity domain of $\Exp$ at $\mathcal{V}_{\alpha}$. In addition, its proof states in~\eqref{eq:aX_with_cos_X0} that
\begin{equation*}
   a(\mathcal{X}_t) 
   \geq  C_0 := \cos(\dist(\mathcal{X}_{0},\mathcal{V}_{\alpha})) > 0,
\end{equation*}
which implies that the function $f$ is weakly-quasi-convex at every $\mathcal{X}_t$ with constant $2 C_0$. Hence
\begin{equation}\label{eq:weak-quasi-conv_with_Delta_t}
2 C_0  \Delta_t \leq \langle \grad f(\mathcal{X}_t) , - \Log_{\mathcal{X}_t} (\mathcal{V}_{\alpha}) \rangle,
\end{equation}
where we defined
\begin{equation*}
    \Delta_t := f(\mathcal{X}_t) - f^*.
\end{equation*}

Similar to the proof of Theorem \ref{thm:exponential_conv}, by the hypothesis on the step size $\eta_t$, Lemma~\ref{lem:GD convergence 1 step} shows that $\mathcal{X}_{t+1}$ is in the injectivity domain of $\Exp$ at $\mathcal{X}_t$. 
Hence, by the definition of Riemannian steepest descent, we have
\begin{equation}\label{eq:Log_SD}
    \Log_{\mathcal{X}_t} (\mathcal{X}_{t+1})=-\eta \grad f(\mathcal{X}_t).
\end{equation}
In addition, the smoothness property~\eqref{eq:quadratic_upper_bound} of $f$ gives
\begin{equation*}
    \Delta_{t+1}-\Delta_t \leq \langle \grad f(\mathcal{X}_t), \Log_{\mathcal{X}_t} (\mathcal{X}_{t+1}) \rangle+\frac{\gamma}{2} \dist^2(\mathcal{X}_t,\mathcal{X}_{t+1}).
\end{equation*}
Substituting~\eqref{eq:Log_SD}, we obtain
\begin{equation}\label{eq:conv_SD_delta_zero_diff_Delta}
    \Delta_{t+1}-\Delta_t \leq \left(-\eta +\frac{\gamma}{2} \eta^2 \right) \| \grad f(\mathcal{X}_t)\|^2 \leq 0,
\end{equation}
since $\eta \leq C_0/\gamma$ with $0 < C_0:= \cos(\dist(\mathcal{X}_0, \mathcal{V}_{\alpha})) \leq 1$ and $\gamma > 0$.

Since $\Gr(n,k)$ has nonnegative sectional curvature, Lemma~\ref{lem:geo_triangle_nonneg} implies
\begin{equation*}
    \dist^2(\mathcal{X}_{t+1}, \mathcal{V}_{\alpha}) \leq \dist^2(\mathcal{X}_t, \mathcal{X}_{t+1})+ \dist^2(\mathcal{X}_t, \mathcal{V}_{\alpha})-2 \langle \Log_{\mathcal{X}_t} (\mathcal{X}_{t+1}), \Log_{\mathcal{X}_t} (\mathcal{V}_{\alpha}) \rangle.
\end{equation*}
Substituting~\eqref{eq:Log_SD} into the above and rearranging terms gives
\begin{equation*}
    2 \eta \langle \grad f(\mathcal{X}_t) , - \Log_{\mathcal{X}_t} (\mathcal{V}_{\alpha}) \rangle \leq \dist^2(\mathcal{X}_t, \mathcal{V}_{\alpha})-\dist^2(\mathcal{X}_{t+1}, \mathcal{V}_{\alpha})+ \eta^2 \| \grad f(\mathcal{X}_t) \|^2.
\end{equation*}
Combining with~\eqref{eq:weak-quasi-conv_with_Delta_t}, we  get
\begin{equation}\label{eq:bound_Delta_t}
    \Delta_t \leq \frac{1}{4 C_0 \eta} ( \dist^2(\mathcal{X}_t, \mathcal{V}_{\alpha})-\dist^2(\mathcal{X}_{t+1}, \mathcal{V}_{\alpha})) + \frac{\eta}{4 C_0} \| \grad f(\mathcal{X}_t) \|^2.
\end{equation}
Now multiplying \eqref{eq:conv_SD_delta_zero_diff_Delta} by $\frac{1}{C_0}$ and summing with~\eqref{eq:bound_Delta_t} gives
\begin{multline}\label{eq:diff_Delta_intermediate}
    \frac{1}{C_0} \Delta_{t+1} - \left( \frac{1}{ C_0} - 1  \right) \Delta_t  \leq  \frac{1}{4 C_0 \eta} ( \dist^2(\mathcal{X}_t, \mathcal{V}_{\alpha})-\dist^2(\mathcal{X}_{t+1}, \mathcal{V}_{\alpha})) \\  
    +\frac{1}{C_0} \left( -\eta + \frac{\gamma}{2}\eta^ 2 + \frac{\eta}{4} \right)
    \| \grad f(\mathcal{X}_t) \|^2.
\end{multline}
By assumption $\eta \leq C_0/\gamma$, where $0 < C_0:= \cos(\dist(\mathcal{X}_0, \mathcal{V}_{\alpha})) \leq 1$ and $\gamma > 0$. Since
\begin{equation*}
    \frac{\eta}{C_0} \left( -1 + \frac{\gamma}{2} \eta + \frac{1}{4}  \right) \leq \frac{1}{\gamma} \left(\frac{C_0}{2} -\frac{3}{4} \right) \leq - \frac{1}{4 \gamma}< 0.
\end{equation*}
the inequality~\eqref{eq:diff_Delta_intermediate} can be simplified to
\begin{equation*}
     \frac{1}{C_0} \Delta_{t+1} - \left( \frac{1}{ C_0} - 1  \right) \Delta_t \leq \frac{1}{4 C_0 \eta} ( \dist^2(\mathcal{X}_t, \mathcal{V}_{\alpha})-\dist^2(\mathcal{X}_{t+1}, \mathcal{V}_{\alpha})),
\end{equation*}
Summing from $0$ to $t-1$ gives
\[
    \frac{1}{C_0} \Delta_t + \sum_{s=1}^{t-1} \Delta_s - \left( \frac{1}{C_0} -1  \right) \Delta_0 \leq  \frac{1}{4 C_0 \eta} \left( \dist^2(\mathcal{X}_0, \mathcal{V}_{\alpha}) - \dist^2(\mathcal{X}_t, \mathcal{V}_{\alpha}) \right).
\]
From the smoothness property~\eqref{eq:quadratic_upper_bound} at the critical point $\mathcal{V}_\alpha$ of $f$, we get 
\[
 \Delta_0 \leq \frac{\gamma}{2} \dist^2(\mathcal{X}_0, \mathcal{V}_{\alpha}).
\]
Combining these two inequalities then leads to
\begin{align*}
    \frac{1}{C_0} \Delta_t + \sum_{s=0}^{t-1} \Delta_s & \leq \frac{1}{C_0}  \Delta_0 + \frac{1}{4 C_0 \eta} \dist^2(\mathcal{X}_0, \mathcal{V}_{\alpha}) \\ 
    & \leq \frac{1}{2C_0} \left(\gamma +\frac{1}{2 \eta}\right) \dist^2(\mathcal{X}_0, \mathcal{V}_{\alpha}).
\end{align*}
Since \eqref{eq:conv_SD_delta_zero_diff_Delta} holds for all $t \geq 0$, it also implies $\Delta_t \leq \Delta_s$ for all $1 \leq s \leq t$. Substituting
\[
 t \Delta_t \leq \sum_{s=0}^{t-1} \Delta_s
\]
into the inequality from above, 
\begin{equation*}
    \Delta_t \leq \frac{1}{2C_0} \frac{\gamma+\frac{1}{2 \eta}}{\frac{1}{C_0}+t} \dist^2(\mathcal{X}_0, \mathcal{V}_{\alpha}) = \frac{\gamma+\frac{1}{2\eta}}{2(C_0 t+1)} \dist^2(\mathcal{X}_0, \mathcal{V}_{\alpha}),
\end{equation*}
we obtain the desired result.
\qed \end{proof}

\begin{remark}
This type of result is standard for functions that are geodesically convex (see, e.g. \cite{zhangFirstorderMethodsGeodesically2016}). Our objective function does not satisfy this property, 
but we can still have a similar upper bound on the iteration complexity for convergence in function value. We note that this does not imply convergence of the iterates to a specific $k$-dimensional subspace, but only convergence of a subsequence of the sequence of the iterates.
\end{remark}

\subsection{Sufficiently small step sizes}

The convergence results in Theorems~\ref{thm:exponential_conv} and~\ref{thm:convex_conv} require that the initial subspace $\mathcal{X}_0$ lies within a distance strictly less than $\pi/2$ from a global minimizer $\mathcal{V}_{\alpha}$. While this condition is independent from the eigengap (unlike results that rely on standard convexity, see appendix), it is also not fully satisfactory: it is hard to verify in practice, and it is unnecessarily severe in numerical experiments. In fact, this condition is  only used to obtain a uniform lower bound on the weak-quasi-convexity constant $a(\mathcal{X}_t) = \theta_k^{(t)} / \tan(\theta_k^{(t)})$ with $\theta_k^{(t)}$ the largest principal angle between $\mathcal{X}_t$ and  $\mathcal{V}_{\alpha}$. Since the Riemannian distance is the $\ell_2$ norm of the principal angles, a contraction in this distance leads automatically to $\theta_k^{(t)} < \pi/2$ if $\theta_k^{(0)} < \pi/2$. If one could guarantee by some other reasoning that $\theta_k^{(t)} $ does not increase after one step, the condition $\dist(\mathcal{X}_0, \mathcal{V}_\alpha) < \pi/2$ would not be needed.


We now show that for sufficiently small step sizes $\eta_t$, the largest principal angle $\theta_k^{(t)} $ between $\mathcal{X}_t$ and  $\mathcal{V}_{\alpha}$ does indeed not increase after each iteration of Riemannian steepest descent regardless of the initial subspace $\mathcal{X}_0$.  While it does not explain what we observe in numerical experiments where large steps can be taken, it is a first result in explaining why we can initialize the iteration at a random initial subspace $\mathcal{X}_0$.

\begin{tcolorbox}
\begin{Proposition} \label{prop:suff_small}
     Riemannian steepest descent started from a subspace $\mathcal{X}_t$ returns a subspace $\mathcal{X}_{t+1}$ such that
     \begin{equation*}
           \theta_k(\mathcal{X}_{t+1},\mathcal{V}_{\alpha}) \leq \theta_k(\mathcal{X}_t,\mathcal{V}_{\alpha}),
     \end{equation*}
     for all step sizes $0 \leq \eta \leq \bar \eta$ where  $ \bar \eta > 0$ is sufficiently small.
\end{Proposition}
\end{tcolorbox}

For the proof of this proposition, we will need the derivatives of certain singular values. While this is well known for isolated singular values, it is possible to generalize to higher multiplicities as well by relaxing the ordering and sign of singular values~\cite{bunse-gerstnerNumericalComputationAnalytic1991}. For a concrete formula, we use the following result from Lemma A.5 in~\cite{lippertFixingTwoEigenvalues2005}.

\begin{lemma}\label{lem:perturb_sing_value}
Let $\sigma_1 \geq \cdots \geq \sigma_n$ be the singular values of $S$ $\in \mathbb{R}^{n \times n}$ with $u_1, \ldots, u_n$ and $v_{1}, \ldots, v_{n}$ the associated left and right orthonormal singular vectors. Suppose that $\sigma_j$ has multiplicity $m$, that is,
$$
 \sigma_{j_0-1} > \sigma_{j_0} = \cdots = \sigma_j =\cdots = \sigma_{j_0+m-1} > \sigma_{j_0+m}.
$$
Then, the $j$th singular value of $S+\eta $T satisfies
\[
 \sigma_j(S+\eta T) = \sigma_j + \eta \lambda_{j - j_0 + 1} + \bigO(\eta^2), \quad \eta \to 0^+,
\]
where $\lambda_{j}$ is the $j$th largest eigenvalue of $\tfrac{1}{2}(U^T B V + V^T B^T U)$
with
\[
 U = \begin{bmatrix} u_{j_0} & \cdots & u_{j_0+m-1} \end{bmatrix} \quad \text{and}\quad V = \begin{bmatrix} v_{j_0} & \cdots & v_{j_0+m-1} \end{bmatrix}.
\]
\end{lemma}




\begin{proof}[of Proposition~\ref{prop:suff_small}]
For ease of notation, let $X:= X_t$ and $X_+ := X_{t+1}$ such that $\mathcal{X}_t = \myspan(X)$ and $\mathcal{X}_{t+1} = \myspan(X_+)$. By definition of the exponential map on Grassmann, the next iterate of the Riemannian SD iteration~\eqref{eq:GD} with step $\eta$ satisfies
\begin{equation*}
    X_+=X V \cos(\eta \Sigma) V^T + U \sin(\eta \Sigma) V^T
\end{equation*}
where
\begin{equation*}
    U \Sigma V^T = -\grad f(\mathcal{X}_t).
\end{equation*}
Since $V$ is orthogonal, we can write
\begin{equation*}
    U \sin(\eta \Sigma) V^T = U (\eta \Sigma) V^T V \left(\frac{\sin(\eta \Sigma)}{\eta \Sigma} \right) V^T = -\eta \grad f(\mathcal{X}_t) V \left(\frac{\sin(\eta \Sigma)}{\eta \Sigma} \right) V^T
\end{equation*}
where $1/\Sigma:=\Sigma^{-1}$ and $\frac{\sin 0}{0} = 1$. 
Taking Taylor expansions of $\sin$ and $\cos$, 
\begin{align*}
   V \cos(\eta \Sigma) V^T &= V \left(I - \bigO(\eta^2) \right) V^T = I-\bigO(\eta^2) \\
   V \frac{\sin(\eta \Sigma)}{\eta \Sigma} V^T &= V \left(I - \bigO(\eta^2) \right) V^T = I-\bigO(\eta^2),
\end{align*}
we obtain
\begin{align}\label{eq:V_a_X_t_plus_1}
    {V}_{\alpha}^T X_+ & = {V}_{\alpha}^T X (I-\bigO(\eta^2)) + {V}_{\alpha}^T (- \eta \grad f(\mathcal{X})) (I-\bigO(\eta^2)) \notag \\ 
    & = {V}_{\alpha}^T (X- \eta \grad f(\mathcal{X}_t)) (I-\bigO(\eta^2))
\end{align}
since $\|V_\alpha\|_2 = \| X \|_2=1$.

Let now $\theta$ be the vector of $k$ principal angles between $\mathcal{X}_t$ and  $\mathcal{V}_\alpha$. As in~\eqref{eq:SVD_Va_X} and~\eqref{eq:SVD_Vb_X}, we therefore have the SVDs
\begin{equation}\label{eq:Va_X_Vb_X}
 V_\alpha^T X = U_1 \cos \theta\, V_1^T \qquad \text{and} \qquad V_\beta^T X = \tilde U_2 \sin \theta\, V_1^T,
\end{equation}
where $U_1,V_1 \in \mathbb{R}^{k \times k}$  and $\tilde U_2 \in \mathbb{R}^{(n-k) \times k}$ have orthonormal columns. Next, we write~\eqref{eq:V_a_X_t_plus_1} in terms of
\[
 M :=\sin^2 \theta \, U_1^T \Lambda_{\alpha} U_1 \cos \theta -  \cos \theta \sin \theta \, \tilde U_2^T \Lambda_{\beta} \tilde U_2 \sin \theta.
\]
Since $\grad f(\mathcal{X}_t) = -2 (I-XX^T)AX$, the identity~\eqref{eq:grad_with_Va} gives
\[
{V}_{\alpha}^T (X- \eta \grad f(\mathcal{X}_t)) = V_\alpha^T X +2 \eta \Lambda_\alpha V_\alpha^T X - 2 \eta V_\alpha^T X X^T A  X.
\]
After substituting~\eqref{eq:XAX worked out} and~\eqref{eq:Va_X_Vb_X}, a short calculation using $\cos^2 \theta = I - \sin^2 \theta$ and the orthogonality of $U_1$ and $V_1$ then shows
\[
V_{\alpha}^T (X- \eta \grad f(\mathcal{X}_t)) = U_1 (\cos \theta + 2\eta M) V_1 ^T.
\]
Relating back to~\eqref{eq:V_a_X_t_plus_1}, we thus obtain
\begin{align*}
    V_{\alpha}^T {X}_{+} &= U_1 (\cos \theta + 2\eta M) V_1 ^T (I-\bigO(\eta^2)) \\
    &= U_1 (\cos \theta + 2\eta M) (I - V_1 ^T \bigO(\eta^2) V_1)V_1 ^T \\
    &= U_1 (\cos \theta + 2\eta M - \bigO(\eta^2)) V_1^T.
\end{align*}
The singular values of $V_{\alpha}^T {X}_{+}$ are therefore the same as the singular values of the matrix $\cos \theta + 2\eta M +\bigO(\eta^2)$.

By  Weyl's inequality (see, e.g., \cite[Cor.~7.3.5]{hornMatrixAnalysis2012a}), each singular value of $\cos \theta + 2\eta M +\bigO(\eta^2)$ is $\bigO(\eta^2)$ close to some singular value of $\cos \theta + 2\eta M$. Let $1 \leq j \leq k$. Denote the $j$th singular value of $\cos \theta + 2\eta M $ by $\sigma_j(\eta)$ to which we will apply Lemma~\ref{lem:perturb_sing_value}. Let $m$ be the multiplicity of $\sigma_j(0)$. Hence, there exists $j_0$ such that $\sigma_{j_0}(0) = \cdots = \sigma_{j}(0) = \cdots = \sigma_{j_0+m-1}(0)$. Since $\cos \theta$ is a diagonal matrix with decreasing diagonal, its $\ell$th singular value equals $\cos \theta_\ell$ and its associated left/right singular vector is the $\ell$th canonical vector $e_\ell$.  
Denoting
\[
 E = \begin{bmatrix} e_{j_0} & \cdots & e_{j_0+m-1} \end{bmatrix}, 
\]
observe that $\cos \theta \, E = \cos \theta_{j_0} \, E$ (here, $\cos \theta$ is a diagonal matrix and $\cos \theta_{j_0} $ is a scalar) and likewise for $\sin \theta \, E$. We thus get
\[
 E^T M E = \sin^2 \theta_{j_0} \cos \theta_{j_0} (U_1^T \Lambda_{\alpha} U_1  -  \tilde U_2^T \Lambda_{\beta} \tilde U_2 ).
\]
In the proof of Proposition~\ref{prop:weak-quasi-convexity}, we showed that the matrix in brackets above is symmetric and positive semi-definite (see~\eqref{eq:La_min_Lb_is_PSD}). Since $0 \leq \theta_{j_0} \leq \pi/2$, the eigenvalues of $E^T M E$ are therefore all non-negative. Lemma~\ref{lem:perturb_sing_value} thus gives that $\sigma_j(\eta) \geq \sigma_j$ for sufficiently small and positive $\eta $. Since the singular values of $V_\alpha^T X_+$ are the cosines of the principal angles between $\mathcal{V}_\alpha$ and $\mathcal{X}_{t+1}$ with step size $\eta \geq 0$, we conclude that there exists $\bar \eta>0$ such that for all $\eta \in [0,\bar\eta]$ it holds
\begin{equation*}
    \theta_j(\mathcal{X}_{t+1}, \mathcal{V}_{\alpha}) \leq \theta_j(\mathcal{X}_t, \mathcal{V}_{\alpha}).
\end{equation*}
Since $j$ was arbitrary, this finishes the proof.
\qed \end{proof}

\section{Numerical experiment}

We report on a small numerical experiment to verify the convergence rates proven above. The steepest descent iteration with fixed step size was implemented in \textsc{Matlab} using the geodesic formula~\eqref{eq:formula_geo}. 

As first test matrix, we took the standard 3D Laplacian on a unit cube, discretized with finite differences and zero Dirichlet boundary conditions. The size of the matrix $A$ is $n=400$. We tested a few values for the block size $k$. They are depicted in the table below, together with other parameters that are relevant for Theorem~\ref{thm:exponential_conv}.

\begin{center}
\begin{tabular}{cccc}\toprule
$k$ & $\delta$ & $\textrm{dist}(X_0, V_\alpha)$ \\ \midrule
1 & $0.0665\ldots$ & $0.113\ldots$ \\
6 & $0.0665\ldots$ & $0.280\ldots$ \\
10 & $0.0262\ldots$ & $0.350\ldots$  \\\bottomrule
\end{tabular}
\end{center}

In Figure~\ref{fig:laplacian}, the convergence of the Riemannian distance is visible in addition to the theoretical convergence rate of Theorem~\ref{thm:exponential_conv}. We see that in all cases, these bound on the convergence are valid (in particular, exponential) although they are rather conservative.

\begin{figure}
\centering
\includegraphics[width=0.9\textwidth]{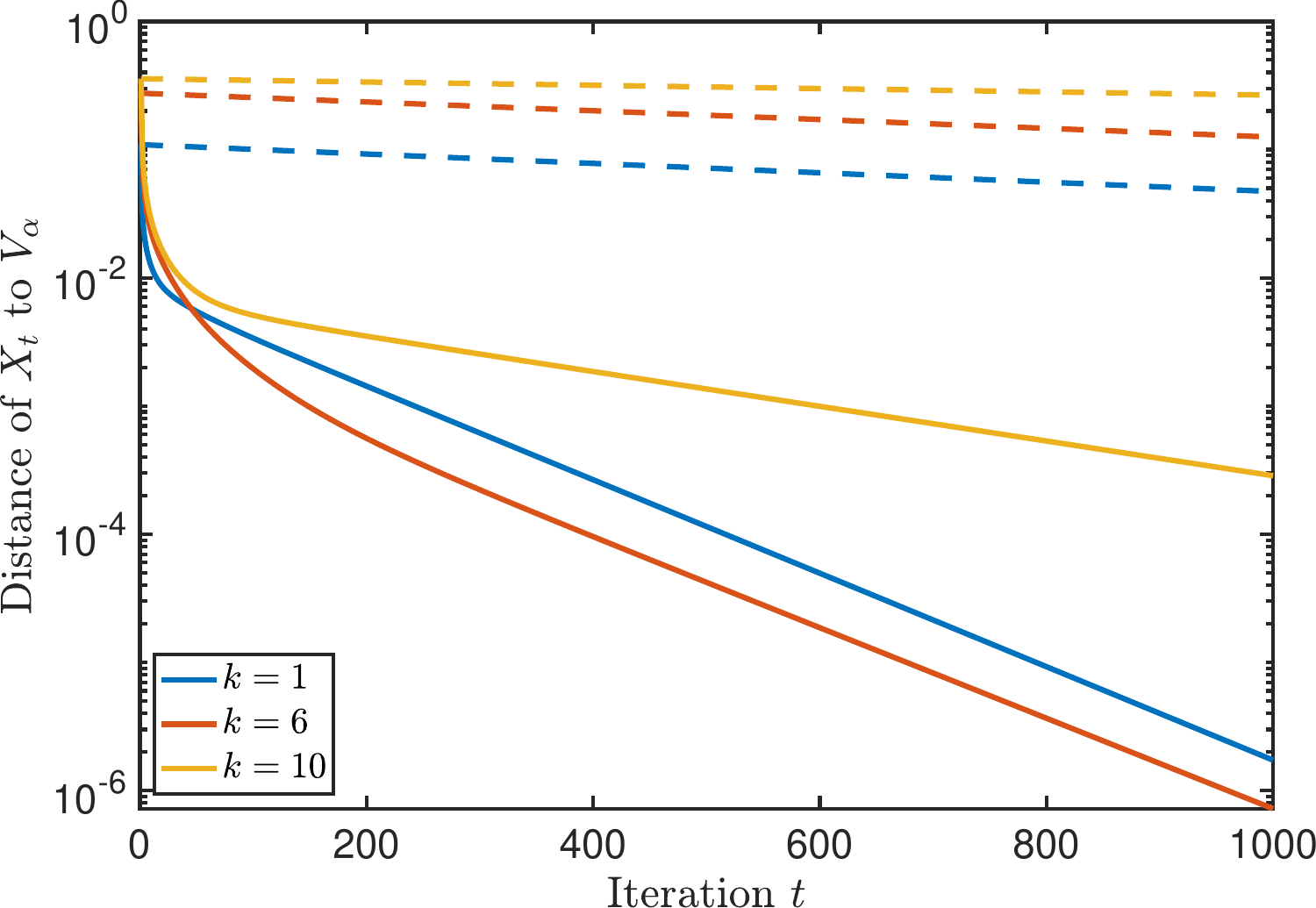}
\caption{Steepest descent along geodesics for the block Rayleigh quotient of size $k$ applied to a discretized 3D Laplacian matrix. The full lines correspond to the experimental values and the dashed lines to the theoretical upper bounds.} \label{fig:laplacian}
\end{figure}

In the second test, we investigate the convergence when the eigengap $\delta$ is small or zero. In particular, we take $A=VDV^T \in \mathbb{R}^{1000 \times 1000}$ with $V$ a random orthogonal matrix and $D$ contains the eigenvalues 
\[
 \lambda_1= 3,\  \lambda_2=2,\  \lambda_3 = 1+10^{-2}+10^{-6},  \ \lambda_4=1+10^{-6},\ \lambda_5=\lambda_6= 1.
\]
The other eigenvalues are equidistantly distributed between $0.1$ and $0.2$. The block size and other relevant parameters for the test are described below. Since the convergence for small $\delta$ slows down considerably after the first 5 iterations, we apply the bounds of Theorem~\ref{thm:convex_conv} at iteration $t=6$ (and treat this as the start with $t=0$).

\begin{center}
\begin{tabular}{ccccc}\toprule
$k$ & $\delta$ & $\textrm{dist}(X_0, V_\alpha)$ & $\textrm{dist}(X_6, V_\alpha)$\\ \midrule
2 & $0.99\ldots$ & $0.051\ldots$  & $0.001\ldots$ \\
3 & $10^{-2}$ & $0.055\ldots$  & $0.031\ldots$ \\
4 & $10^{-6}$ & $0.063\ldots$  & $0.045\ldots$ \\
5 & $0$ & $0.070\ldots$  & $0.054\ldots$  \\\bottomrule
\end{tabular}
\end{center}

The convergence in function value is visible in Figure~\ref{fig:tinygap}. Observe that we have displayed a logarithmic scale for both axes whereas before the figure had a logarithmic scale only for $y$-axis. Algebraic convergence like $1/t$ is therefore visible as a straight line. We see in the figure that the convergence is not easily described, and that there is no clear difference between zero or small gap. However, the upper bounds of Theorem~\ref{thm:convex_conv} are again valid. In addition, when the gap is not small, the convergence is clearly faster.

\begin{figure}
\centering
\includegraphics[width=0.9\textwidth]{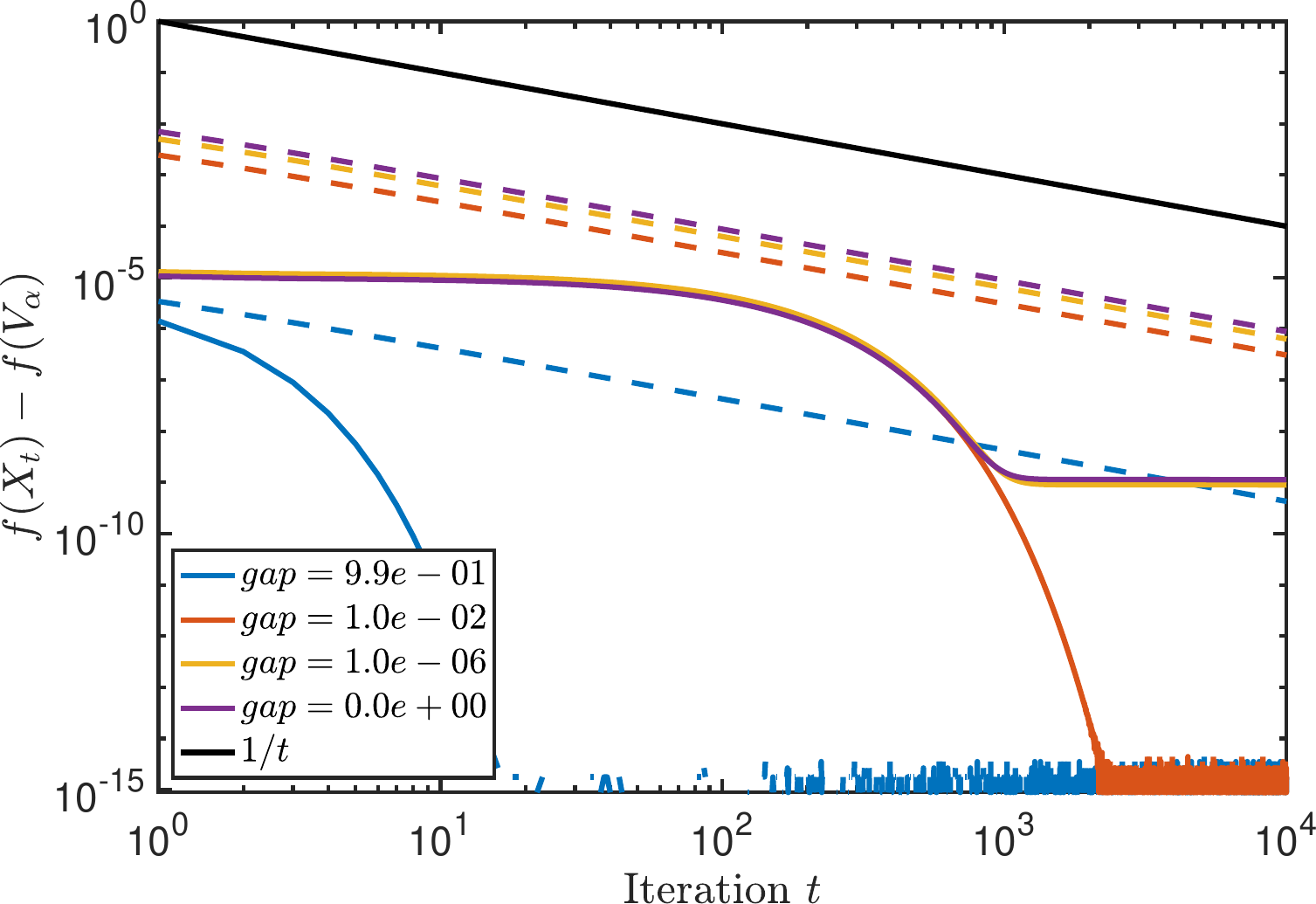}
\caption{Steepest descent along geodesics for the block Rayleigh quotient of size $k$ applied to a random matrix with small eigengaps. The full lines correspond to the experimental values and the dashed lines to the theoretical upper bounds. Each color corresponds to a certain eigengap $\delta$.} \label{fig:tinygap}
\end{figure}

\section{Conclusion and future work}

We provided the first non-asymptotic convergence rates for Riemannian steepest descent on the Grassmann manifold for computing a subspace spanned by $k$ leading eigenvectors of a symmetric matrix $A$. 

Our main idea was to exploit a convexity-like structure of the block Rayleigh quotient, which can be of much more general interest than for only analyzing steepest descent. One example is line search methods, which have usually favourable properties compared to vanilla steepest descent. Also, weakly-quasi-convex functions have been proven to admit accelerated algorithms \cite{nesterov2020primal}, while accelerated or almost accelerated Riemannian algorithms have been developed in \cite{zhang2018towards,alimisis2020continuous,alimisis2021momentum}. It would naturally be interesting to examine whether a provable accelerated method  can be developed for the block Rayleigh quotient on the Grassmann manifold. This would hopefully reduce the dependence of the iteration complexity on the eigengap $\delta$ from $\bigO(1/\delta)$ to $\bigO(1/\sqrt{\delta})$.

Another interesting direction is to extend the analysis of \cite{alimisis2021distributed} from the  computation of just one leading eigenvector to computation of a whole subspace, using the generalized machinery developed in this work, or develop a noisy version of steepest descent and compare with noisy power method \cite{hardt2014noisy}.

\bibliographystyle{plain} 
\bibliography{refs}

\appendix
\section{Geodesic convexity} \label{sec:convexity}

Let $\delta > 0$ and thus $\mathcal{V}_\alpha$ is the unique minimizer of $f$.  Define the following neighbourhood of $\mathcal{V}_\alpha$ in $\Gr(n,k)$:
\begin{equation}\label{eq:geodesic_region_N}
 N_*(\varphi) = \{ \mathcal{X} \in \Gr(n,k) \colon \theta_k(\mathcal{X}, \mathcal{V}_\alpha) < \varphi \} \qquad \text{with $\varphi \in [0, \pi/4]$}.
\end{equation}
Here, $\theta_k(\mathcal{X}, \mathcal{V}_\alpha)$ denotes the largest principal angle between $\mathcal{X}$ and $\mathcal{V}_\alpha$. Since $\theta_k$ is a metric on $\Gr(n,k)$ (see~\cite{qiuUnitarilyInvariantMetrics2005}), any two subspaces $\mathcal{X},\mathcal{Y} \in N_*(\varphi)$ will satisfy $\theta_k(\mathcal{X},\mathcal{Y}) < \pi /2$ by triangle inequality. They thus have a unique connecting geodesic. It is shown in \cite[Lemma~2]{ahn2021riemannian} that for any fixed $\varphi \in [0, \pi/4]$ this geodesic remains in $N_*(\varphi)$. Each set $N_*(\varphi)$ is thus an open totally geodesically convex set as defined in, e.g., \cite[Def.~11.16]{boumal2022intromanifolds}. 

One of the main results in \cite{ahn2021riemannian}, namely Cor.~4, states that $f$ is geodesically convex on $N_*(\pi/4)$. This is unfortunately wrong and we present a small counterexample.

\paragraph{Counterexample for Cor.~4 in \cite{ahn2021riemannian}.}

Here we use the notation of \cite{ahn2021riemannian}. The reader is encouraged to take a look there for notational purposes.

Take $c:=\cos(\pi/4) = \sqrt{2}/2$ and $0 \le \varepsilon<1$. Define the matrices
$$
X_p := \begin{pmatrix} 1 & 0 \\ 0 & 1 \\ 0 & 0  \\ 0 & 0 \end{pmatrix}, \quad
U_p := \begin{pmatrix} c & 0 \\ 0 & c \\ c & 0  \\ 0 & c \end{pmatrix}, \quad
M := U_p \begin{pmatrix} 1 & 0 \\ 0 & \varepsilon \end{pmatrix}.
$$
These matrices satisfy the conditions posed in \cite{ahn2021riemannian}:
\begin{itemize}
    \item Principal alignment: $X_p^T U_p = \begin{pmatrix} c & 0 \\ 0 & c \end{pmatrix}$.
    \item Principal angles between $X_p$ and $U_p$ are in $[0,\pi/4]$.
    \item $U = U_p$ since $Q=I$.
\end{itemize}

Now consider the following tangent vector of unit Frobenius norm:
$$
 \Delta = \begin{pmatrix} 0 & 0 \\ 0 & 0 \\ 0 & 1  \\ 0 & 0 \end{pmatrix}.
$$
It is clearly a tangent vector of $[X_p]$ since $X_p^T \Delta = 0$.
The Hessian of $f_{full}$ at $[X_p]$ in the direction of $\Delta$ satisfies (see equation (4.2) in \cite{ahn2021riemannian})
$$
\textnormal{Hess}f_{full}([X_p])[\Delta,\Delta] = -2  \Tr(M^T \Delta \Delta^T (I-X_p X_p^T) M) + \|(\Delta X_p^T + X_p\Delta ^T) M \|_F^2.
$$
Simple calculation shows that
$$
 \textnormal{Hess}f_{full}([X_p])[\Delta,\Delta] = -2 c^2 + (1+\varepsilon^2) c^2.
$$
Hence for $\varepsilon<1$, we have $\textnormal{Hess}f_{full}([X_p])[\Delta,\Delta]<0$ and the $f_{full}$ is non-convex which is in contrast with Corollary 4. \qed

\vspace{8mm}
Instead, our Theorem~\ref{eq:g-convex_domain} guarantees convexity when $\varphi$ depends on the spectral gap.
Since $f$ is smooth, the function is geodesically convex on $N_*(\varphi)$ if and only if its Riemannian Hessian is positive definite on $N_*(\varphi)$; see, e.g., \cite[Thm.~11.23]{boumal2022intromanifolds}. We will therefore compute the eigenvalues of $\Hess f$ based on its matrix representation. This requires us to first vectorize the tangent space.

From~\eqref{eq:def_TXGr}, a matrix $G$ is a tangent vector if and only if $G^TX = 0$. Hence, taking $X_\perp \in \mathbb{R}^{n \times (n-k)}$ orthonormal such that $\mathcal{X}^\perp=\myspan(X_\perp)$, we have the equivalent definition
\[
 T_X \Gr(n,k) = \{ X_\perp M \colon M \in \mathbb{R}^{(n-k) \times k} \}.
\]

The matrix $M$ above can be seen as the coordinates of $G=X_\perp M $ in the basis $X_\perp$. More specifically, by using the linear isomorphism $\vecop \colon \mathbb{R}^{n \times k} \to \mathbb{R}^{nk}$ that stacks all columns of a matrix under each other, we can define the tangent vectors of $\Gr(n,k)$ as standard (column) vectors in the following way:
\[
 \vecop(G) = \vecop(X_\perp M ) = (I_k \otimes X_\perp) \vecop(M).
\]
Here, the Kronecker product $\otimes$ appears due to \cite[Lemma 4.3.1]{hornTopicsMatrixAnalysis1991}. By well-known properties of $\otimes$  (see, e.g., \cite[Chap.~4.2]{hornTopicsMatrixAnalysis1991}), the matrix $I_k \otimes X_\perp$ has orthonormal columns. We have thus obtained an orthonormal basis for the (vectorized) tangent space. With this setup, we can now construct the Hessian. 
 
\begin{lemma}\label{lem:matrix_Hx}
Let $I_k \otimes X_\perp$ be the orthonormal basis for the vectorization of $T_{\mathcal{X}} \Gr(n,k)$. Then the Riemannian Hessian of $f$ at $\mathcal{X}$ in that basis has the symmetric matrix representation
\begin{equation}\label{eq:matrix_Hx}
 H_X = 2 (X^T A X \otimes I_{n-k} - I_k \otimes X_\perp^T A X_\perp).
\end{equation}
Furthermore, with $1 \leq i \leq k$ and $1 \leq j \leq n-k$ its $k(n-k)$ eigenvalues satisfy
\[
 \lambda_{i,j}(H_X) = 2(\lambda_i(X^T A X) - \lambda_j(X_\perp^T A X_\perp)).
\]
\end{lemma}
\begin{proof}
Since $\vecop$ is a linear isomorphism, the symmetric matrix $H_X$ satisfies
\[
 \Hess f(X)[X_\perp M, X_\perp M] = \langle \vecop(M), H_X \vecop(M) \rangle,  \qquad \forall M \in \mathbb{R}^{n \times (n-k)},
\]
where $\langle \cdot, \cdot \rangle$ is the Euclidean inner product. Define $m = \vecop(M)$. Plugging in the formula~\eqref{eq:Hessian_f_inner_product} for $\Hess f$, we calculate
\begin{align*}
 \Hess f(X)[X_\perp M, X_\perp M] &= 2\langle X_\perp M, X_\perp M X^T A X - AX_\perp M \rangle \\
 &= 2\langle (I \otimes X_\perp) m,  (X^T A X \otimes X_\perp)  m  - (I \otimes AX_\perp) m  \rangle \\
 &= 2\langle m,  (I \otimes X_\perp)^T (X^T A X \otimes X_\perp   - I \otimes AX_\perp) m  \rangle \\
 &= 2\langle m,  (X^T A X \otimes I   - I \otimes X_\perp^T AX_\perp) m  \rangle
\end{align*}
Here, we used typical calculus rules for the Kronecker product (see, e.g., \cite[Chap.~4.2]{hornTopicsMatrixAnalysis1991}). We recognize the matrix $H_X$ directly.

The eigenvalues of~\eqref{eq:matrix_Hx} can be directly obtained using  \cite[Thm.~4.4.5]{hornTopicsMatrixAnalysis1991}.
\qed \end{proof}

Taking $X=V_\alpha$ and $X_\perp = V_\beta$, Lemma~\ref{lem:matrix_Hx} shows immediately that the minimal eigenvalue of $\Hess f(\mathcal{V}_\alpha)$ is equal to $2\delta = 2(\lambda_k - \lambda_{k+1})$. Since $\delta > 0$, $\Hess f$ will remain strictly positive definite in a neighbourhood of $\mathcal{V}_\alpha$ by continuity. To quantify this neighbourhood, we will connect $\mathcal{V}_\alpha$ to an arbitrary $\mathcal{X}$ using a geodesic and see how this influences the bounds of Lemma~\ref{lem:matrix_Hx}. This also requires connecting $\mathcal{V}_\beta$ to $\mathcal{X}^\perp$. The next lemma shows that both geodesics are closely related. Recall that $\sin(t\theta)$ and  $\cos(t\theta)$ denote diagonal matrices of size $k \times k$. For convenience, we will denote by $O$ a zero matrix whose dimensions are clear from the context and is not always square.

\begin{lemma}\label{lem:geodesic_perp}
Let $X,Y \in \mathbb{R}^{n \times k}$ be such that $X^T X = Y^T Y = I_k$ with $k \leq n/2$. Denote the principal angles between $\myspan(X)$ and $\myspan(Y)$ by $\theta_1 \leq \cdots \leq \theta_k$ and assume that $\theta_k< \pi/2$. Choose $X_\perp, Y_\perp \in \mathbb{R}^{n \times (n-k)}$ such that $X_\perp^T X_\perp = Y_\perp^T Y_\perp = I_{n-k}$ and $\myspan(X_\perp) = \myspan(X)^\perp$, $\myspan(Y_\perp) = \myspan(Y)^\perp$. Define the curves
\begin{align*}
 \gamma(t) &\colon [0,1] \to \mathbb{R}^{n \times k} , &  t &\mapsto  X V_1 \cos(t\theta) + X_\perp V_2  \begin{bmatrix} O \\ \sin(t\theta) \end{bmatrix},  \\
 \gamma_\perp(t) &\colon [0,1] \to \mathbb{R}^{n \times (n-k)} , & t &\mapsto X_\perp  V_2 \begin{bmatrix} I \\ & \cos(t\theta)   \end{bmatrix} - X V_1  \begin{bmatrix} O & \sin(t\theta) \end{bmatrix},
\end{align*}
where the orthogonal matrices $V_1,V_2$ are the same as in Lemma~\ref{lemma:CS_square_blocks}. Then $\myspan(\gamma(t))$ is the connecting geodesic on $\Gr(n,k)$ from $\myspan(X)$ to $\myspan(Y)$. Likewise, $\myspan(\gamma_\perp(t))$ is a connecting geodesic on $\Gr(n,n-k)$  from $\myspan(X_\perp)$ to $\myspan(Y_\perp)$. Furthermore, $\gamma(t)$ and $\gamma_\perp(t)$ are orthonormal matrices for all $t$.
\end{lemma}




\begin{proof}
Assume $\theta_1 = \cdots = \theta_r = 0$, where $r=0$ means that $\theta_1 > 0$.  Like in the proof of Prop.~\ref{prop:weak-quasi-convexity}, the CS decomposition of $X$ and $Y$ from Lemma~\ref{lemma:CS_square_blocks} can be written in terms of their principal angles $\theta_1, \ldots, \theta_k$. Since $\theta_k < \pi/2$ and $n \leq k/2$, this gives after dividing certain block matrices the relations
\begin{align*}
 Y^T X &= U_1 \, \cos(\theta) \, V_1^T,  & 
 Y^T X_\perp &= U_1 \begin{bmatrix}O_{k \times (n-2k)} & \sin(\theta) \end{bmatrix} V_2^T  \\
Y_\perp^T X &= U_2 \begin{bmatrix}O_{(n-2k) \times k} \\ \sin(\theta) \end{bmatrix} V_1^T, & Y_\perp^T X_\perp &= U_2  \begin{bmatrix}-I_{n-2k}  \\ & -\cos(\theta) \end{bmatrix} V_2^T,
\end{align*}
where  $U_1, V_1$ and $U_2,V_2$ are orthogonal matrices of size $k \times k$ and $(n-k) \times (n-k)$, resp.




Denote $\mathcal{X}=\myspan(X)$ and $\mathcal{Y}=\myspan(Y)$. By definition, the connecting geodesic $\gamma(t)$ is determined by the tangent vector $ \Log_{\mathcal{X}}(\mathcal{Y})$, which can be computed from \eqref{eq:log formula}. To this end, we first need the compact SVD of $M:= X_\perp X_\perp^T Y (X^T Y)^{-1}$. Substituting the results from above, we get (cfr.~\eqref{eq:Log_with_V2})
\[
 M 
 = X_\perp  V_2 \begin{bmatrix} O_{(n-2k) \times k} \\ \sin(\theta)  \end{bmatrix}  U_1^ T U_1  \, (\cos(\theta))^{-1}\,  V_1^T = 
 X_\perp  V_2 \begin{bmatrix} O_{(n-2k) \times k} \\ I_k \end{bmatrix} \,  \tan(\theta) \, V_1^T.
\]

Observe that this is a compact SVD. Applying \eqref{eq:log formula}, we therefore get
\[
 G := \Log_{\mathcal{X}}(\mathcal{Y}) =  U \Sigma V^T \quad \text{with $U = X_\perp V_2 \begin{bmatrix} O \\ I_k \end{bmatrix}, \ \Sigma = \theta, \ V = V_1$}
\]
and from~\eqref{eq:formula_geo}, the connecting geodesic satisfies
\[
 \Exp_{\mathcal{X}}(tG) = \myspan( \, X V_1 \cos(t\theta) + X_\perp V_2 \begin{bmatrix} O \\ I_k \end{bmatrix} \sin(t\theta) \, ).
\]
We have proven the stated formula for $\gamma(t)$. Verifying that $\gamma(t)^T \gamma(t) = I_k$ follows from a simple calculation that uses $\cos^2(t \theta) + \sin^2(t \theta) = I_k$.

Denote $\mathcal{X}^\perp=\myspan(X_\perp)$ and $\mathcal{Y}^\perp=\myspan(Y_\perp)$. To prove $\gamma_\perp(t)$, we proceed similarly by computing $G^\perp:= \Log_{\mathcal{X}^\perp}(\mathcal{Y}^\perp)$, which requires now the SVD of $M^\perp:= X X^T Y_\perp (X_\perp^T Y_\perp)^{-1}$. Again substituting the results from the CS decomposition, we get
\begin{align*}
 M^\perp &= X V_1 \begin{bmatrix}O_{k \times (n-2k)} & \sin(\theta)  \end{bmatrix} U_2^ T U_2 \begin{bmatrix} -I_{n-2k} \\ & - \cos(\theta) \end{bmatrix}^{-1}  V_2^T \\
 &= X V_1 \begin{bmatrix}O_{k \times (n-2k)} & -\tan(\theta)  \end{bmatrix} V_2^T
\end{align*}
Since \eqref{eq:log formula} requires a compact SVD with a \emph{square} $\Sigma$, we rewrite this as
\[
 M^\perp = \begin{bmatrix} \widetilde X & X V_1 \end{bmatrix} \begin{bmatrix}O_{(n-2k) \times (n-2k)} \\ & -\tan(\theta)  \end{bmatrix} V_2^T
\]
where $\widetilde X$ contains $n-2k$ columns that are orthonormal to $X$ (the final result will not depend on $\widetilde X$). Let $\theta^\perp_1 \leq \cdots \leq  \theta^\perp_{n-k}$ denote the principal angles between $\mathcal{X}^\perp$ and $\mathcal{Y}^\perp$. Up to zero angles, they are the same as those between $\mathcal{X}$ and $\mathcal{Y}$. Since $k \leq n/2$, we thus have
\[
 \theta^\perp_1 = \cdots =  \theta^\perp_{n-2k} = 0, \  \theta^\perp_{n-2k+1} = \theta_1, \ldots,   \theta^\perp_{n-k} =  \theta_k.
\]
Applying \eqref{eq:log formula} with these principal angles, we obtain
\[
 G^\perp := \Log_{\mathcal{X}^\perp}(\mathcal{Y^\perp}) =  U \Sigma V^T \quad \text{with $U = -\begin{bmatrix} \widetilde X & X V_1 \end{bmatrix} , \ \Sigma = \theta^\perp, \ V = V_2$}.
\]
From \eqref{eq:formula_geo}, the corresponding geodesic satisfies
\begin{align*}
 \Exp_{\mathcal{X}^\perp}(tG^\perp) &= \myspan( \, X_\perp   V_2 \cos(t\theta^\perp) - \begin{bmatrix} \widetilde X & X V_1 \end{bmatrix} \sin(t\theta^\perp) \, ) \\
 &= \myspan( \, X_\perp   V_2 \begin{bmatrix} I_{n-2k}  \\ & \cos(t\theta) \end{bmatrix}  - \begin{bmatrix}  O_{n \times (n-2k)} & X V_1 \sin(t\theta) \end{bmatrix}  \, ). 
\end{align*}
Rewriting the block matrix, we have proven $\gamma_\perp(t)$. Its orthonormality is again a straightforward verification.
\qed \end{proof}

With the previous lemma, we can now investigate the Riemannian Hessian of $f$ near $\mathcal{V}_\alpha$ when it is given in the matrix form $H_X$ of Lemma~\ref{lem:matrix_Hx}. Let $\mathcal{X} = \myspan(X) \in \Gr(n,k)$ with orthonormal $X$. Its principal angles with $\mathcal{V}_\alpha$ are $\theta_1 \leq \cdots \leq \theta_k < \pi/2$.
Use the substitutions $X \mapsto V_\alpha, Y \mapsto X$ and $X_\perp \mapsto V_\beta, Y_\perp \mapsto X_\perp$ in Lemma~\ref{lem:geodesic_perp} to define the geodesics $\gamma(t)$ and $\gamma_\perp(t)$ that connect $\mathcal{V}_\alpha$ to $\mathcal{X}$, and $\mathcal{V}_\beta$ to $\mathcal{X}^\perp$, resp. Denoting
\[
 C := \cos(\theta), \ S := \sin(\theta), 
 \ \widetilde C := \begin{bmatrix} I \\ & C \end{bmatrix}, 
 \ \widetilde S := \begin{bmatrix} O \\  S \end{bmatrix},
\]
we get the following expressions for the geodesics:
\[
 \gamma(t) = V_\alpha V_1 C + V_\beta V_2 \widetilde S, \quad 
 \gamma_\perp(t) = V_\beta V_2 \widetilde C - V_\alpha V_1 \widetilde S^T.
\]
Recall that $H_X$ is defined using $X^T A X$ and $X_\perp^T A X_\perp$. Since $\gamma(1) = XQ_1$ and $\gamma_\perp(1) = X^\perp Q_2$ for some orthogonal matrices $Q_1, Q_2$, we can write with $A = V_\alpha \Lambda_{\alpha} V_\alpha^T + V_\beta \Lambda_{\beta} V_\beta^T$ that

\begin{equation}\label{eq:XtAX_XptAXp_geodesic}
\begin{aligned}
 Q_1^T X^T A X Q_1 &=  \gamma(1)^T A \gamma(1) \\
 &=  C \, (V_1^T \Lambda_{\alpha} V_1) \, C  +  \widetilde S^T \, (V_2^T \Lambda_{\beta} V_2) \, \widetilde S  \\
 Q_2^TX_\perp^T A X_\perp Q_2&=  \gamma_\perp(1)^T A \gamma_\perp(1) \\
  &=  \widetilde C \, (V_2^T \Lambda_{\beta} V_2) \, \widetilde C + \widetilde S \, (V_1^T \Lambda_{\alpha} V_1) \, \widetilde S^T.
\end{aligned}
\end{equation}

Here we used simplifications like $V_\beta^T AV_\alpha = V_\beta^T V_\alpha \Lambda_{\alpha} = 0$.


A simple bounding of the eigenvalues of the difference of these matrices results in the main result.
\begin{theorem}\label{eq:g-convex_domain}
Let $k \leq n/2$. Define the neighbourhood
\[
 B_* = \left\{ \mathcal{X} \in \Gr(n,k) \colon \sin^2 (\theta_k(\mathcal{X}, \mathcal{V}_\alpha)) \leq \frac{\delta}{\lambda_1 + \lambda_k} \right\},
\]
then $f$ is geodesically convex on $B_*$.
\end{theorem}

\begin{proof}
Our aim is to show that $\lambda_{i,j}(H_X)$ remains positive given the bound on $\theta_k$. From Lemma~\ref{lem:matrix_Hx}, we see that 
\begin{equation}\label{eq:condition_Hess_pos}
 \lambda_{\min}(H_X) \geq 0 \quad \iff \quad \lambda_{\min}(X^T A X) \geq \lambda_{\max}(X_\perp^T A X_\perp).
\end{equation}
Since $Q_1,Q_2$ are orthogonal in~\eqref{eq:XtAX_XptAXp_geodesic}, it suffices to find a lower and upper bound of, resp.,
\begin{align*}
 \lambda_{\min}(X^T A X) &= \lambda_{\min}(C \, (V_1^T \Lambda_{\alpha} V_1) \, C  +  \widetilde S^T \, (V_2^T \Lambda_{\beta} V_2) \, \widetilde S) \\
 \lambda_{\max}(X_\perp^T A X_\perp) &= \lambda_{\max}(\widetilde C \, (V_2^T \Lambda_{\beta} V_2) \, \widetilde C + \widetilde S \, (V_1^T \Lambda_{\alpha} V_1) \, \widetilde S^T).
\end{align*}
Standard eigenvalue inequalities for symmetric matrices (see, e.g., \cite[Cor.~4.3.15]{hornMatrixAnalysis2012a}) give
\begin{align*}
 \lambda_{\min}(X^T A X) &\geq \lambda_{\min}(C \, (V_1^T \Lambda_{\alpha} V_1) \, C)  +  \lambda_{\min}(\widetilde S^T \, (V_2^T \Lambda_{\beta} V_2) \, \widetilde S) \\
 \lambda_{\max}(X_\perp^T A X_\perp) &\leq \lambda_{\max}(\widetilde C \, (V_2^T \Lambda_{\beta} V_2) \, \widetilde C)  + \lambda_{\max}(\widetilde S \, (V_1^T \Lambda_{\alpha} V_1) \, \widetilde S^T).
\end{align*}
Recall that $\lambda_1 \geq \cdots \geq \lambda_n$ are the eigenvalues of $A$. 
Since $\widetilde S$ is a tall rectangular matrix, we apply the generalized version of Ostrowski's theorem from~\cite[Thm.~3.2]{highamModifyingInertiaMatrices1998} to each term above\footnote{Observe that the cited theorem orders the eigenvalues inversely to the convention used in this paper.} and obtain 
\begin{align*}
 \lambda_{\min}(C \, (V_1^T \Lambda_{\alpha} V_1) \, C) &\geq \lambda_{\min}(C^2) \lambda_{\min}(\Lambda_{\alpha}) = \cos^2(\theta_k) \lambda_k \\
 \lambda_{\min}(\widetilde S^T \, (V_2^T \Lambda_{\beta} V_2) \, \widetilde S) &\geq \lambda_{\min}(\tilde S^T \tilde S) \lambda_{\min}(\Lambda_{\beta}) = \sin^2(\theta_1) \lambda_n,
\end{align*} 
since the matrices $V_1,V_2$ are orthogonal and $\theta_1 \leq \cdots \leq \theta_k < \pi/2$. Adding this gives the lower bound
\begin{equation}\label{eq:convex_intermediate_lower_bound}
\lambda_{\min}(X^T A X) \geq \cos^2(\theta_k) \lambda_k + \sin^2(\theta_1) \lambda_n \geq \cos^2(\theta_k) \lambda_k.
\end{equation}
Likewise, using the block structure of $\widetilde S$, we get
\begin{align*} 
 \lambda_{\max}(\widetilde C \, (V_2^T \Lambda_{\beta} V_2) \, \widetilde C) &\leq \lambda_{\max}(C^2) \lambda_{\max}(\Lambda_{\beta}) = \cos^2(\theta_1) \lambda_{k+1} \\
 \lambda_{\max}(\widetilde S \, (V_1^T \Lambda_{\alpha} V_1) \, \widetilde S^T) &= 
 \lambda_{\max}(S \, (V_1^T \Lambda_{\alpha} V_1) \, S) \\
 &\leq \lambda_{\max}(S^2) \lambda_{\max}(\Lambda_{\alpha}) = \sin^2(\theta_k) \lambda_1
\end{align*}
and thus
\begin{equation}\label{eq:convex_intermediate_upper_bound}
 \lambda_{\max}(X_\perp^T A X_\perp) \leq \cos^2(\theta_1) \lambda_{k+1} +  \sin^2(\theta_k) \lambda_1 \leq  \lambda_{k+1} +  \sin^2(\theta_k) \lambda_1.
\end{equation}
The condition~\eqref{eq:condition_Hess_pos} is thus satisfied when
\[
 \cos^2(\theta_k) \lambda_k = \lambda_k - \sin^2(\theta_k) \lambda_k \geq  \lambda_{k+1} +  \sin^2(\theta_k) \lambda_1,
\]
which reduces to the bound on $\theta_k$ in the statement of the theorem.

It remains to show that $B_*$ is an open totally geodesically convex set. Since $\lambda_1 \geq \lambda_k \geq \lambda_{k+1} \geq 0$, we get
\[
\frac{\lambda_k - \lambda_{k+1}}{\lambda_1 + \lambda_{k}} \leq \frac{\lambda_k }{2\lambda_{k}} = \frac{1}{2}.
\]
Hence, $B_* = N_*(\varphi)$ with $\varphi \leq \pi/4$ since $\sin^2(\pi/4) = 1/2$.
\qed \end{proof}

If $k=1$, the proof above can be simplified.

\begin{corollary}\label{cor:g-convex_domain_sphere}
Let $k=1$ and define the neighbourhood
\[
 B_* = \left\{ \mathcal{X} \in \Gr(n,1) \colon \sin^2 (\theta_1(\mathcal{X}, \mathcal{V}_\alpha)) \leq \frac{\delta}{\delta + \lambda_1 - \lambda_n} \right\}.
\]
Then $f$ is geodesically convex on $B_*$.
\end{corollary}
\begin{proof}
Since $k=1$, there is no need to simplify the bounds~\eqref{eq:convex_intermediate_lower_bound} and \eqref{eq:convex_intermediate_upper_bound} as was done above. This gives that $f$ is convex as long as
\[
 \cos^2(\theta_1) \lambda_1 + \sin^2(\theta_1) \lambda_n \geq  \cos^2(\theta_1) \lambda_{2} +  \sin^2(\theta_1) \lambda_1.
\]
Rewriting leads directly to the stated condition on  $\sin^2(\theta_1)$.
\qed \end{proof}
Remark that optimizing $f$ on $\Gr(n,1)$ is equivalent to
\begin{equation}\label{eq:min_f_sphere}
 \min_{x \in \mathbb{R}^n} - x^T A x \qquad \text{s.t.} \quad \|x\| = 1,
\end{equation}
which is the minimization of the Rayleigh quotient problem on the unit sphere $S^{n-1} = \{ x \in \mathbb{R}^n \colon x^T x = 1 \}$. 
Cor.~\ref{cor:g-convex_domain_sphere} can therefore also be phrased in terms of a geodesically convex region for this problem. Denoting a unit norm top eigenvector of $A$ by $v_1$ and using that $\sin^2 \theta_1 = 1 - \cos^2 \theta_1$, we get that~\eqref{eq:min_f_sphere} is geodesically convex on 
\[
 \hat B_* = \left\{ x \in S^{n-1} \colon (x^T v_1)^2 \geq 1 -  \frac{\delta}{\delta + \lambda_1 - \lambda_n} \right\}.
\]
This result can now be directly compared to \cite[Lemma 7]{pmlr-v119-huang20e} where the corresponding region is defined as $(x^T v_1)^2 \geq 1 -  \frac{\delta}{\delta + \lambda_1}$. This is a stricter condition and our result is therefore a small improvement.

\section{Convergence of Steepest Descent with step $\frac{1}{\gamma}$}
\label{sec:big_step}

We now prove convergence of steepest descent with a more tractable choice of step-size compared to the analysis of the main paper. However, this requires a slightly better initialization at most $\frac{\pi}{2 \sqrt{2}}$ away from the minimizer.

\subsection{Maximum extent of the iterates}
We first prove that steepest descent with step-size at most $\frac{1}{\gamma}$ does not guarantee contraction on distances from step to step, but we can still bound the distance at step $t$ with the initial distance up to a scalar:
\begin{Proposition} \label{prop:big_step_distance}
     Consider steepest descent applied to $f$ with step-size $\eta \leq \frac{1}{\gamma}$. If the iterates $\mathcal{X}_t$ satisfy $\theta_k(\mathcal{X}_t,\mathcal{V}_{\alpha})<\frac{\pi}{2}$, then they also satisfy
     \begin{equation*}
         \textnormal{dist}^2(\mathcal{X}_t, \mathcal{V}_{\alpha}) \leq 2 \textnormal{dist}^2(\mathcal{X}_0, \mathcal{V}_{\alpha}).
     \end{equation*}
\end{Proposition}
\begin{proof}
Consider the discrete Lyapunov function
\begin{equation*}
    \mathcal{E}(t)= \frac{1}{\gamma} (f(\mathcal{X}_t)-f^*)+\frac{1}{2} \textnormal{dist}^2(\mathcal{X}_t, \mathcal{V}_{\alpha}).
\end{equation*}
Then
\begin{equation*}
    \mathcal{E}(t+1)-\mathcal{E}(t) = \frac{1}{\gamma}(f(\mathcal{X}_{t+1})-f(\mathcal{X}_t))+\frac{1}{2} ( \textnormal{dist}^2(\mathcal{X}_{t+1},\mathcal{V}_{\alpha})-\textnormal{dist}^2(\mathcal{X}_t,\mathcal{V}_{\alpha}) ).
\end{equation*}
By $\gamma$-smoothness of $f$, we have
\begin{equation*}
    f(\mathcal{X}_{t+1})-f(\mathcal{X}_t) \leq \langle \textnormal{grad}f(\mathcal{X}_t),\textnormal{Log}_{\mathcal{X}_t}(\mathcal{X}_{t+1}) \rangle +\frac{\gamma}{2} \textnormal{dist}(\mathcal{X}_t,\mathcal{X}_{t+1})^2= \left( -\eta +\frac{\gamma}{2} \eta^2 \right) \|\textnormal{grad}f(\mathcal{X}_t) \| ^2.
 \end{equation*}
 We also know by Proposition \ref{prop:weak-quasi-convexity} that
 \begin{equation*}
     \langle \textnormal{grad}f(\mathcal{X}), -\Log_{\mathcal{X}}(\mathcal{V_{\alpha})} \rangle \geq 0,
 \end{equation*}
 for any $\mathcal{X}$ with $\theta_k(\mathcal{X},\mathcal{V}_\alpha) < \pi/2$. \newline
 By the fact that the sectional curvatures of the Grassmann manifold are non-negative, we have
 \begin{align*}
     \textnormal{dist}^2(\mathcal{X}_{t+1},\mathcal{V}_{\alpha}) & \leq \textnormal{dist}^2(\mathcal{X}_t,\mathcal{V}_{\alpha}) + \textnormal{dist}^2(\mathcal{X}_{t+1},\mathcal{X}_t) - 2 \langle \Log_{\mathcal{X}_{t}}(\mathcal{X}_{t+1}), \Log_{\mathcal{X}_t}(\mathcal{V_{\alpha})} \rangle \\ & = \textnormal{dist}^2(\mathcal{X}_t,\mathcal{V}_{\alpha}) + \eta^2 \| \textnormal{grad}f(\mathcal{X}_t) \|^2 + 2 \eta \langle \textnormal{grad}f(\mathcal{X}_t) , \Log_{\mathcal{X}_t}(\mathcal{V_{\alpha})} \rangle \\ & \leq \textnormal{dist}^2(\mathcal{X}_t,\mathcal{V}_{\alpha}) + \eta^2 \| \textnormal{grad}f(\mathcal{X}_t) \|^2.
 \end{align*}
 Thus
 \begin{align*}
     \mathcal{E}(t+1)-\mathcal{E}(t) \leq \left(-\frac{\eta}{\gamma}+\frac{\eta^2}{2} \right)\| \textnormal{grad}f(\mathcal{X}_t) \| ^2+ \frac{\eta^2}{2}\| \textnormal{grad}f(\mathcal{X}_t) \|^2 \leq \left(-\frac{\eta}{\gamma}+\eta^2 \right) \| \textnormal{grad}f(\mathcal{X}_t) \|^2 \leq 0,
 \end{align*}
 because $\eta \leq \frac{1}{\gamma}$. \newline
 Since $\mathcal{E}(t)$ does not increase, we have
 \begin{align*}
     \frac{1}{2} \textnormal{dist}^2(\mathcal{X}_t,\mathcal{V}_{\alpha}) & \leq \mathcal{E}(t) \leq \mathcal{E}(0) =  \frac{1}{\gamma} (f(\mathcal{X}_0)-f^*)+\frac{1}{2} \textnormal{dist}^2(\mathcal{X}_0, \mathcal{V}_{\alpha}) \\ & \leq \frac{1}{2} \textnormal{dist}^2(\mathcal{X}_0, \mathcal{V}_{\alpha}) + \frac{1}{2} \textnormal{dist}^2(\mathcal{X}_0, \mathcal{V}_{\alpha})= \textnormal{dist}^2(\mathcal{X}_0, \mathcal{V}_{\alpha})
 \end{align*}
 and the desired result follows.
\qed \end{proof}

\subsection{Convergence under positive eigengap}
When $\delta>0$, we can use gradient dominance to prove convergence of steepest descent to the (unique) minimizer in terms of function values:
\begin{Proposition}
Steepest descent with step-size $\eta = \frac{1}{\gamma}$ initialized at $\mathcal{X}_0$ such that
\begin{equation*}
    \textnormal{dist}(\mathcal{X}_0,\mathcal{V}_{\alpha}) \leq \frac{\pi}{4}
\end{equation*}
satisfies
\begin{equation*}
    f(\mathcal{X}_t)-f^* \leq \left(1-0.32 c_Q \frac{\delta}{\gamma} \right)^t (f(\mathcal{X}_0)-f^*).
\end{equation*}
\end{Proposition}
\begin{proof}
By the previous result and an induction argument to guarantee that the biggest angle between $\mathcal{X}_t$ and $\mathcal{V}_{\alpha}$ stays strictly less than $\pi/2$, we can bound the quantities $a(\mathcal{X}_t)$ uniformly from below:
\newline Since $\textnormal{dist}(\mathcal{X}_t,\mathcal{V}_{\alpha}) \leq \sqrt{2} \cdot \textnormal{dist}(\mathcal{X}_0,\mathcal{V}_{\alpha}) \leq \frac{\sqrt{2} \pi}{4}$, we have
\begin{equation*}
    a(\mathcal{X}_t) \geq \cos(\theta_k(\mathcal{X}_t,\mathcal{V}_{\alpha})) \geq \cos(\textnormal{dist}(\mathcal{X}_t,\mathcal{V}_{\alpha})) \geq \cos\left(\frac{\sqrt{2} \pi}{4} \right) \geq 0.4.
\end{equation*}
By $\gamma$-smoothness of $f$, we have
\begin{equation*}
f(\mathcal{X}_{t+1})-f(\mathcal{X}_t) \leq -\frac{\| \textnormal{grad}f(\mathcal{X}_t) \| ^2}{2 \gamma}
\end{equation*}
and applying gradient dominance (Proposition \ref{prop:PL}), we get the bound
\begin{equation*}
    f(\mathcal{X}_{t+1})-f(\mathcal{X}_t) \leq -\frac{2 c_Q \delta a^2(\mathcal{X}_t)}{\gamma}(f(\mathcal{X}_t)-f^*)
\end{equation*}
thus
\begin{equation*}
    f(\mathcal{X}_{t+1})-f^* \leq \left(1-2 c_Q a^2(\mathcal{X}_t ) \frac{\delta}{\gamma}\right ) (f(\mathcal{X}_t)-f^*) \leq  \left(1-0.32 c_Q \frac{\delta}{\gamma} \right) (f(\mathcal{X}_t)-f^*).
\end{equation*}
By induction the desired result follows.
\qed \end{proof}
We now state the iteration complexity of steepest descent algorithm:
\begin{theorem}
Steepest descent with step-size $\frac{1}{\gamma}$ starting from a subspace $\mathcal{X}_0$ with distance at most $\frac{\pi}{4}$ from $\mathcal{V}_{\alpha}$ computes an estimate $\mathcal{X}_T$ of $\mathcal{V}_{\alpha}$ such that $\textnormal{dist}(\mathcal{X}_T,\mathcal{V}_{\alpha})\leq \epsilon$  in at most 
\begin{equation*}
    T=\bigO\left( \frac{\gamma}{\delta} \log \frac{f(\mathcal{X}_0)-f^*}{\delta \epsilon} \right).
\end{equation*}
\end{theorem}
\begin{proof}
For $\textnormal{dist}(\mathcal{X}_T,\mathcal{V}_{\alpha})< \epsilon$, it suffices to have
\begin{equation*}
    f(\mathcal{X}_T)-f^* \leq c_Q \epsilon^2 \delta
\end{equation*}
by quadratic growth of $f$ in Proposition \ref{prop:quadratic growth}. Using $(1-c)^T \leq \exp(-c T)$ for all $T \geq 0$ and $0 \leq c \leq 1$, the previous result gives that it suffices to choose $T$ as the smallest integer such that 
\begin{equation*}
    f(\mathcal{X}_T)-f^* \leq \exp\left(- 0.32 c_Q \frac{\delta}{\gamma} T \right) (f(\mathcal{X}_0)-f^*) \leq c_Q \epsilon^2 \delta.
\end{equation*}
Solving for $T$ and substituting $c_Q = 4/\pi^2$, we get the required statement.
\qed \end{proof}

\subsection{Gap-less result}

We also prove a convergence result for the function values when $\delta$ is assumed to be $0$:

\begin{theorem}
    Steepest descent with step-size $\eta=\frac{1}{\gamma}$
    initialized at $\mathcal{X}_0$ such that
\begin{equation*}
    \textnormal{dist}(\mathcal{X}_0,\mathcal{V}_{\alpha}) \leq \frac{\pi}{4}
\end{equation*}
satisfies
\begin{equation*}
    f(\mathcal{X}_t)-f^* \leq \frac{f(\mathcal{X}_0)-f^*+\frac{\gamma}{2}\textnormal{dist}^2(\mathcal{X}_0, \mathcal{V}_{\alpha})}{0.4 t + 1} = \bigO\left(\frac{1}{t}\right). 
\end{equation*}
\end{theorem}

\begin{proof}
By Proposition \ref{prop:big_step_distance}, we have that $\textnormal{dist}(X_t,\mathcal{V}_{\alpha}) \leq \frac{\sqrt{2} \pi}{4}$ and $f$ satisfies the weak-quasi-convexity inequality at any iterate $\mathcal{X}_t$ of steepest descent with constant $C_0:=0.4$.

Consider the discrete Lyapunov function
\begin{equation*}
    \mathcal{E}(t)= \frac{C_0t+1}{\gamma} (f(\mathcal{X}_t)-f^*)+\frac{1}{2} \textnormal{dist}^2(\mathcal{X}_t, \mathcal{V}_{\alpha}).
\end{equation*}
We have that
\begin{align*}
    \mathcal{E}(t+1)-\mathcal{E}(t) = & \frac{C_0t+C_0+1}{\gamma}(f(\mathcal{X}_{t+1})-f^*)-\frac{C_0t+1}{\gamma} (f(\mathcal{X}_t)-f^*) 
    \\ & +\frac{1}{2}(\textnormal{dist}^2(\mathcal{X}_{t+1},\mathcal{V}_{\alpha})-\textnormal{dist}^2(\mathcal{X}_t,\mathcal{V}_{\alpha})).
\end{align*}
 Now we have to estimate a bound for $\textnormal{dist}^2(\mathcal{X}_{t+1},\mathcal{V}_{\alpha})-\textnormal{dist}^2(\mathcal{X}_t,\mathcal{V}_{\alpha})$.
 By $\gamma$-smoothness of $f$ and denoting 
 $\Delta_t=f(\mathcal{X}_t)-f^*$ we have
 \begin{equation*}
    \Delta_{t+1}-\Delta_t \leq \langle \textnormal{grad}f(\mathcal{X}_t),\textnormal{Log}_{\mathcal{X}_t}(\mathcal{X}_{t+1}) \rangle +\frac{\gamma}{2} \textnormal{dist}^2(\mathcal{X}_t,\mathcal{X}_{t+1})= -\frac{\|\textnormal{grad}f(\mathcal{X}_t) \|^2}{2\gamma}
 \end{equation*}
 By $C_0$-weak-strong-convexity of $f$ and the fact that the Grassmann manifold is of positive curvature, we have
 \begin{equation*}
     C_0\Delta_t \leq \frac{\gamma}{2}(\textnormal{dist}^2(\mathcal{X}_t,\mathcal{V}_{\alpha})-\textnormal{dist}^2(\mathcal{X}_{t+1},\mathcal{V}_{\alpha}))+\frac{ \| \textnormal{grad}f(\mathcal{X}_t) \|^2}{2\gamma}
 \end{equation*}
 
 Summing this to the previous inequality, we get
 \begin{equation*}
    \textnormal{dist}^2(\mathcal{X}_{t+1},\mathcal{V}_{\alpha})-\textnormal{dist}^2(\mathcal{X}_t,\mathcal{V}_{\alpha}) \leq \frac{2}{\gamma} ((1-C_0)(f(\mathcal{X}_t)-f(\mathcal{X}_{t+1}))-C_0(f(\mathcal{X}_{t+1})-f^*)).
 \end{equation*}
 Thus
 \begin{align*}
     \mathcal{E}(t+1)-\mathcal{E}(t) & \leq \frac{C_0t+ 1 }{\gamma} (f(\mathcal{X}_{t+1})-f(\mathcal{X}_t))+\frac{C_0}{\gamma} (f(\mathcal{X}_{t+1})-f^*) \\ & + \frac{1 - C_0}{\gamma} (f(\mathcal{X}_t)-f(\mathcal{X}_{t+1}))- \frac{C_0}{\gamma} (f(\mathcal{X}_{t+1})-f^*) 
     \\&=\frac{C_0t+C_0}{\gamma} (f(\mathcal{X}_{t+1})-f(\mathcal{X}_t)) \leq 0.
 \end{align*}

Thus $\mathcal{E}(t) \leq \mathcal{E}(0)$ and the result follows.
 
\qed \end{proof}

\end{document}